\documentclass[11pt,a4paper,twoside]{article}
\usepackage[margin=24mm]{geometry}
\usepackage{amsmath,amssymb,amsthm,mathrsfs,mathtools,hyperref,fancyhdr,xcolor}
\usepackage{booktabs,longtable,array,enumitem}
\hypersetup{hidelinks}
\allowdisplaybreaks[2]
\numberwithin{equation}{section}
\setlist{nosep,leftmargin=2em}
\theoremstyle{plain}
\newtheorem{theorem}{Theorem}[section]
\newtheorem{lemma}[theorem]{Lemma}
\newtheorem{proposition}[theorem]{Proposition}
\newtheorem{corollary}[theorem]{Corollary}
\theoremstyle{definition}
\newtheorem{definition}[theorem]{Definition}
\theoremstyle{remark}

\newcommand{\dd}{\mathrm{d}}
\newcommand{\ii}{\mathrm{i}}
\DeclareMathOperator{\Ric}{Ric}
\DeclareMathOperator{\Vol}{Vol}

\DeclareMathOperator{\Id}{Id}

\hypersetup{pdftitle={Explicit Ricci-flat metrics on a special Schoen threefold},pdfsubject={Private research draft; not for public release},pdfauthor={}}
\title{Explicit Ricci-flat metrics on a special Schoen threefold}
\author{Jixiang Fu, Jiuhe Liu, Yunyang Xiao and Shing-Tung Yau}

\date{}

\begin{document}

\maketitle
\begin{abstract}
We give an explicit convergent representation of the Ricci-flat K\"ahler
metric in a one-parameter family of classes on a fixed special Schoen
threefold. The threefold resolves a shifted square-torus quotient with
sixteen disjoint elliptic singular curves. Radial Eguchi--Hanson product
metrics and fixed flat collars define a positive background with an exact
$a^4$ volume defect. Spherical harmonics, tangential Fourier modes,
Frobenius recurrences, radial ordered integrals, and the flat quotient heat
kernel specify the inverse of the background scalar linearized operator.
A boundary transmission series
has contraction factor $1/15$. Independent weighted zeroth-order and
six-dimensional Schauder estimates give the inverse bound needed for the
quadratic and cubic Monge--Amp\`ere terms. The resulting source-degree
recursion converges for $0<a\le2^{-2260}$, with explicit coefficient,
potential-tail, and relative metric-tail bounds. The representation uses
specified kernels, prescribed integrals, and convergent series; its
parameter interval is a sufficient quantitative bound on this fixed model.
\end{abstract}

\section{Introduction}\label{intro:scope}
The Ricci-flat metric in a prescribed K\"ahler class on a compact
Calabi--Yau manifold is determined by a global complex Monge--Amp\`ere
equation \cite{Yau}. We seek a representation of that solution in which the
linear inverse, every nonlinear coefficient, and the convergence bounds
are specified. The present paper carries out this construction on one
fixed special Schoen threefold, for a family of equal-size exceptional
classes approaching a toroidal orbifold limit.

Schoen introduced Calabi--Yau threefolds obtained from fiber products of
rational elliptic surfaces with section \cite{Schoen}.
The special complex structure used here admits a shifted toroidal quotient
description \cite{NibbelinkVaudrevange}. Its covering space is the square
three-torus $E^3$, with $E=\mathbb C/(\mathbb Z+\ii\mathbb Z)$, and its
group is generated by two commuting affine involutions. The quotient has
sixteen disjoint elliptic singular curves; resolving them replaces the
normal $\mathbb C^2/\{\pm1\}$ factors by Eguchi--Hanson spaces.
An affinely equivalent shifted quotient, its Eguchi--Hanson product
gluing, and embedded special Lagrangian tori were studied by Lu
\cite{Lu}; the equivalence is recorded in Section~\ref{geom:section}.
Ricci-flat existence and the local gluing
geometry are therefore established features of this example.
The additional objective here is a specified global scalar inverse and a
quantitatively convergent nonlinear correction.

{The background--Green-operator method is developed for
Kummer K3 surfaces in \cite{KummerK3}. That construction uses
Eguchi--Hanson metrics at sixteen isolated
$\mathbb C^2/\{\pm1\}$ singularities, with an $a^4$ volume defect
and an exactly quadratic Monge--Amp\`ere equation.
The same manuscript also states the completed construction on
collapsing elliptic K3 surfaces using the Gross--Wilson gluing of
semi-flat and Ooguri--Vafa metrics. Section~7 of that manuscript
outlines the elliptic Green recursion and states its convergence;
the detailed elliptic kernel formulas and proofs of the estimates
are omitted there.

The companion Kummer threefold manuscript \cite{KummerCompanion}
uses Calabi's metric at twenty-seven isolated order-three quotient
singularities, with an $a^6$ volume defect and a cubic nonlinear
equation. The present special Schoen model instead has sixteen
disjoint elliptic singular curves, with transverse
$\mathbb C^2/\{\pm1\}$ geometry and Eguchi--Hanson product metrics.
Its volume defect is of order $a^4$, but its complex dimension is
three, so both quadratic and cubic polarizations are required.
The product geometry requires spherical modes in four normal real
directions and Fourier modes in both tangential real directions.
The estimates retain every tangential frequency and use the full
six-dimensional elliptic equation; the inverse is not obtained by
simply taking a product of the K3 inverse.

These companion manuscripts are cited to explain the relation between
the constructions. All kernels, transmission estimates, weighted
bounds, and convergence arguments needed for the present model are
proved in this paper, independently of their results.}

\subsection{The main result and the construction}
Let $D_1,\ldots,D_{16}$ be the exceptional divisors and let $\kappa$ denote
the pullback of the flat orbifold class. The class family is
\begin{equation}\label{intro:class}
 \alpha_a=\kappa-\frac{\pi a^2}{2}\sum_{j=1}^{16}[D_j].
\end{equation}
Each exceptional rational fiber has area $\pi a^2$.
Section~\ref{geom:section} constructs a positive background
$\omega_a^{\mathrm{bg}}\in\alpha_a$ and computes its volume density
$\dd V_a=h_a\mu$, where
$\mu=(\ii/8)\Omega\wedge\overline\Omega$ and $\int_X\mu=1/4$.
The required normalization and source are
\[
 c_a=1-16\pi^2a^4,\qquad \mathcal R_a=c_a/h_a-1.
\]
The source includes a constant normalization term outside the gluing
annuli. The final correction consequently acts on the whole manifold.

\begin{theorem}[Main result]\label{intro:main}
Let $X$, $\alpha_a$, and $\omega_a^{\mathrm{bg}}$ be the fixed special
Schoen model and background constructed in Section~\ref{geom:section}.
Use the capped spaces $E_a,Y_a$ of Definition~\ref{conv:def-norms}, and
set $M=2^{1120}$. For every $0<a\le2^{-2260}$, the specified ordered
recursion \eqref{conv:recursion-first}--\eqref{conv:recursion} produces
$u_a=\sum_{n\ge1}U_{a,n}$ convergent in $E_a$, with
\[
 \omega_a^{\mathrm{CY}}=\omega_a^{\mathrm{bg}}+
 \ii\partial\bar\partial u_a>0,\qquad
 \frac{(\omega_a^{\mathrm{CY}})^3}{6}=(1-16\pi^2a^4)\mu,
 \qquad\int_Xu_a\,\dd V_a=0.
\]
The metric $\omega_a^{\mathrm{CY}}$ is smooth and is the unique
Ricci-flat metric in $\alpha_a$.
For every $n\ge1$ and integer $N\ge0$,
\begin{align*}
 \|U_{a,n}\|_{E_a}&\le1728Ma^4(4/5)^n,\\
 \left\|u_a-\sum_{n=1}^NU_{a,n}\right\|_{E_a}
 &\le8640Ma^4(4/5)^{N+1},\\
 \left\|(\omega_a^{\mathrm{bg}})^{-1}\ii\partial\bar\partial
 \left(u_a-\sum_{n=1}^NU_{a,n}\right)\right\|_\infty
 &\le8640Ma(4/5)^{N+1}.
\end{align*}
Moreover,
\[
 \frac{23}{24}\omega_a^{\mathrm{bg}}<\omega_a^{\mathrm{CY}}
 <\frac{25}{24}\omega_a^{\mathrm{bg}}.
\]
\end{theorem}
The complete quantitative statement and proof are given in
Theorem~\ref{conv:main-theorem}.

The proof separates two convergent series. First, the flat quotient Green
kernel and the local radial solution operators give a boundary transmission
series for the global inverse $G_a$ of the background operator
$L_a=\Delta_{g_a}/2$. Its contraction factor $1/15$ is measured in a
boundary energy norm. Second, the quadratic and cubic polarizations of the
Monge--Amp\`ere determinant give the recursion for $u_a$.
To connect these steps, we prove a global weighted zeroth-order bound and
an independent local Schauder certificate. Combining them in the capped
norms of Definition~\ref{conv:def-norms} yields the norm of $G_a$ used in
the nonlinear argument.

\begin{center}
\begin{tabular}{@{}ll@{}}\toprule
Estimate or datum & Sufficient choice\\\midrule
Tube boundary and flat collar entrance & $b=1/16,\quad c=1/32$\\
Boundary energy contraction & $1/15$\\
Weighted zeroth-order inverse & $M_0=10^8$\\
Local mixed-weight Schauder estimate & $C_{\mathrm{sc}}=2^{1060}$\\
Inverse in the final capped spaces & $M=2^{1120}$\\
Normalized source & $\|\mathcal R_a\|_{Y_a}<864a^4$\\
Quadratic and cubic compositions & $K_2=3Ma^{-3},\quad K_3=Ma^{-6}$\\
Sufficient nonlinear interval & $0<a\le2^{-2260}$\\\bottomrule
\end{tabular}
\end{center}

\subsection{Scope and organization}
An explicit representation here consists of specified local modes, finite
algebra, prescribed integrals, and convergent series with controlled tails.
The expansion is in source degree for each fixed $a>0$; its coefficients
retain their dependence on the geometric scale $a$.
The numerical endpoint certifies a nonempty interval for the stated
background and norms. It does not claim an optimal parameter range or a
finite numerical evaluation of the six-dimensional metric.
All results concern this fixed special complex structure and the family
\eqref{intro:class}.

Section~\ref{geom:section} specifies the quotient, background, class, and
volume defect. Section~\ref{kernel:local-section} constructs the local
spherical--Fourier kernels and radial transfer formulas.
Section~\ref{kernel:global-section} proves the collar estimate and the
global transmission construction. Section~\ref{weighted:section} defines
the capped spaces and proves the weighted inverse bound.
Section~\ref{conv:section} estimates the exact cubic equation and proves
the convergence theorem. Section~\ref{computational:section} explains
truncation and the scope of the numerical certificates.
Appendices~\ref{c0:section} and~\ref{sch:sec:schauder} contain the
zeroth-order and geometric elliptic certificates.

\vspace{2mm}

{\bfseries Acknowledgments.} 
The authors used AI tools to assist in exploring proof strategies
and checking calculations.
The authors are solely responsible for verifying all arguments and calculations and for the final content.

\section{A special Schoen resolution and an explicit K\"ahler background}
\label{geom:section}

We fix a special complex structure and construct a family of K\"ahler classes approaching a toroidal orbifold limit.
This section specifies the quotient, the positive background, its class,
and its exact volume defect. The following sections construct the local
and global linear kernels for this background.
The background is Ricci-flat only on part of the manifold; the subsequent volume correction is allowed to change it everywhere.

\subsection{The torus quotient, sixteen singular curves, and the volume form}
Let
\[
 E=\mathbb C/(\mathbb Z+\ii\mathbb Z),\qquad A=E^3,
\]
and define the following automorphisms of $A$:
\begin{align*}
 \theta(z_1,z_2,z_3)&=(z_1,-z_2,-z_3),\\
 \sigma(z_1,z_2,z_3)&=(-z_1,z_2,-z_3+\tfrac12).
\end{align*}
They generate $G\cong\mathbb Z_2\times\mathbb Z_2$.
Their product is $\theta\sigma(z)=(-z_1,-z_2,z_3-\tfrac12)$,
which equals $\sigma\theta(z)$ modulo the lattice.
Its third component is a half-period translation, so it has no fixed points.
Write $Y=A/G$, and let $\pi\colon X\to Y$ be its crepant resolution.
The resulting threefold is a special Schoen threefold; this quotient description is given by
Groot Nibbelink--Vaudrevange \cite[Sections~2.3, 3.1, and~3.2]{NibbelinkVaudrevange}.

For comparison with Lu's quotient \cite{Lu}, write his generators as
\[
 \alpha(u)=(-u_1+\tfrac12,-u_2+\tfrac12,u_3),\qquad
 \beta_{\rm Lu}(u)=(-u_1,u_2,-u_3).
\]
The affine square-torus isometry
$F(u_1,u_2,u_3)=(u_3,u_2-\tfrac14,u_1-\tfrac14)$ satisfies
$F\alpha F^{-1}=\theta$ and $F\beta_{\rm Lu}F^{-1}=\sigma$ modulo the lattice.
Thus the two presentations describe the same flat quotient.

The sixteen fixed curves of $\theta$ on the covering torus are paired by $\sigma$, and the sixteen fixed curves of $\sigma$ are paired by $\theta$.
Consequently, the singular locus of $Y$ consists of sixteen disjoint elliptic curves $C_j$, with eight in each family.
A neighborhood of each curve is
\[
 E\times\bigl(B^4/\{\pm1\}\bigr),
\]
and its resolution is a truncated neighborhood in $E\times\mathcal O_{\mathbb P^1}(-2)$.
There is no twisting of the normal bundle along $E$.
Indeed, the normal bundle on the covering torus is trivial, and the other generator exchanges two disjoint components.
After choosing one representative, its stabilizer acts as $-\Id$ on the two normal complex coordinates and trivially along the curve.
The normal linear parts of the remaining pairing maps are unitary diagonal sign matrices, so they preserve the radial potentials below.

The forms
\[
 \Omega=\dd z_1\wedge\dd z_2\wedge\dd z_3,\qquad
 \mu=\frac{\ii}{8}\Omega\wedge\overline\Omega
\]
are $G$-invariant, and $\Omega$ extends without zeros over $X$.
In the local resolution coordinates $Q=u^2$ and $\zeta=v/u$, this follows from
\[
 \dd u\wedge\dd v=\tfrac12\dd Q\wedge\dd\zeta.
\]
The unit square torus has volume $1$, and $G$ has order $4$; hence
\begin{equation}
 \int_X\mu=\frac14.
 \label{geom:mu-volume}
\end{equation}
Fix the radii
\begin{equation}
 b=\frac1{16},\quad c=\frac1{32},\quad
 r_1=\frac1{128},\quad r_2=\frac2{128},\quad
 r_3=\frac3{128},\quad r_4=\frac4{128}=c.
 \label{geom:radii}
\end{equation}
The sixteen normal tubes of radius $b$ are disjoint.
The minimum distance between the two families is $1/4$, whereas $2b=1/8$; distances within either family are larger.
Denote the resolved tubes by $U_j$ and their combined boundary by $\Sigma=\coprod_{j=1}^{16}\partial U_j$.
The annular region $[c,b]$ will be a fixed, exactly flat collar.

\subsection{A radial background with prescribed volume density}
Define
\[
 \nu(t)=\begin{cases}\exp(-1/t),&t>0,\\0,&t\le0,\end{cases}
 \qquad H(t)=\frac{\nu(t)}{\nu(t)+\nu(1-t)},
\]
and set
\begin{align}
 \psi(r)&=
 H\left(\frac{r-r_1}{r_2-r_1}\right)
 H\left(\frac{r_4-r}{r_4-r_3}\right),\nonumber\\
 J(r)&=\int_0^r t^3\psi(t)\,\dd t,\qquad I=J(r_4),\qquad
 \beta_a=\frac{a^4}{4I}.
 \label{geom:bump}
\end{align}
Thus $0\le\psi\le1$, $\psi=1$ on $[r_2,r_3]$, and its support is contained in $[r_1,r_4]$.
The plateau and support intervals give the exact bounds
\begin{equation}
 \frac{65}{128^4}\le4I\le\frac{255}{128^4}.
 \label{geom:bump-bounds}
\end{equation}
Throughout this construction, take $0<a\le1/128$, so that $\beta_a\le1/65<1$.

Put $\rho=|u|^2+|v|^2=r^2$, and define
\begin{align}
 D(\rho)&=\rho^2+a^4-4\beta_a J(\sqrt\rho),\nonumber\\
 p(\rho)&=\frac{\sqrt{D(\rho)}}\rho,\qquad
 q(\rho)=p+\rho p'
 =\frac{\rho[1-\beta_a\psi(\sqrt\rho)]}{\sqrt{D(\rho)}}.
 \label{geom:radial-eigenvalues}
\end{align}
Since
\begin{equation}
 D'=2\rho[1-\beta_a\psi(\sqrt\rho)]>0,
 \qquad pq=1-\beta_a\psi(\sqrt\rho),
 \label{geom:radial-density}
\end{equation}
all three functions $D,p,q$ are positive.
Here $D(0)=a^4$, and the inequalities for $p,q$ are understood for
$\rho>0$; smooth positivity at the exceptional divisor is checked below.
Fix the additive constant in the potential by
\begin{equation}
 f(\rho)=
 \begin{cases}
 \displaystyle\rho-\int_\rho^{r_4^2}(p(s)-1)\,\dd s,&0<\rho<r_4^2,\\
 \rho,&\rho\ge r_4^2.
 \end{cases}
 \label{geom:radial-potential}
\end{equation}
On the $j$th tube, set
\[
 \Phi_a=\frac{f(\rho)+|w|^2}{2},\qquad
 \omega_a^{\mathrm{bg}}=\ii\partial\bar\partial\Phi_a,
\]
where $w$ is a local flat coordinate along the elliptic curve.
Outside the tubes, use the flat form inherited from the unit torus quotient.
The construction requires only a prescribed one-dimensional integral.

On the core $r<r_1$, the function $f$ differs by a constant from the Eguchi--Hanson potential
\begin{equation}
 K_a(\rho)=\sqrt{\rho^2+a^4}-a^2\operatorname{arsinh}(a^2/\rho).
 \label{geom:eh-potential}
\end{equation}
For $r\ge c$, one has $D=\rho^2$, so $p=q=1$ and $f=\rho$.
In the resolution chart, $\rho=|Q|(1+|\zeta|^2)$.
Subtracting $a^2\log|Q|$ from the core potential $f$ gives a smooth local potential across the zero section.
This is a pluriharmonic change of potential away from that section, and it provides the smooth extension of the form.
The restriction to the zero section is
\[
 \frac{a^2\ii}{2}\partial\bar\partial\log(1+|\zeta|^2),
\]
so each exceptional fiber has area $\pi a^2$.
{On the core, $K_a-a^2\log\rho=C_a+H_a(\rho^2)$, where $H_a$ is
smooth and $H_a'(0)=1/(4a^2)$. The smooth gauge is therefore
$a^2\log(1+|\zeta|^2)+H_a(|Q|^2(1+|\zeta|^2)^2)$, up to a constant.
At $Q=0$, its half-Hessian has diagonal entries
$(1+|\zeta|^2)^2/(8a^2)$ and $a^2/[2(1+|\zeta|^2)^2]$, with zero
mixed entry; the elliptic entry is $1/2$. Positivity follows, and
the bundle-chart gauges agree as forms by pluriharmonicity.
}

The cutoff vanishes to infinite order at its endpoints, and the radial potentials are invariant under the pairing maps.
They therefore define a global smooth positive K\"ahler form.
This construction does not add the $\partial\bar\partial$ of a global smooth function to the original class: the exceptional fiber areas have changed.

\subsection{The exact K\"ahler class, volume, and source}
Let $D_j=E\times\mathbb P^1$ be the exceptional divisors, and let $\kappa$ be the pullback of the flat orbifold class.
Then
\begin{equation}
 [\omega_a^{\mathrm{bg}}]=\kappa-\frac{\pi a^2}{2}\sum_{j=1}^{16}[D_j].
 \label{geom:kahler-class}
\end{equation}
Indeed, the background agrees with the flat form outside the tubes.
For each tube the relative group $H^2(U_j,\partial U_j;\mathbb R)$
is one-dimensional, generated by the exceptional divisor class.
The flat pullback class has zero pairing with an exceptional rational
fiber, whereas $[D_j]$ has pairing $-2$. The area $\pi a^2$
therefore determines the coefficient $-\pi a^2/2$ in
\eqref{geom:kahler-class}.
The normal bundle of each divisor is $\mathcal O(-2)$ in the $\mathbb P^1$
direction and trivial along $E$; thus
$D_j^3=0$, $\kappa D_j^2=-2$, and $\kappa^2D_j=0$.
The divisors are disjoint, and $\kappa^3/6=1/4$; hence
$\alpha_a^3/6=1/4-4\pi^2a^4$.
{The Thom isomorphism gives
$H^2(U_j,\partial U_j;\mathbb R)\simeq H^0(D_j;\mathbb R)=\mathbb R$;
the relative exact sequence expresses the class difference, supported
in the tubes, in these generators. On $D_j$, the two restrictions are
$[D_j]|_{D_j}=-2e$ and the unit-area elliptic class, giving the stated
intersections. For $t=\pi a^2/2$, the cubic correction is
$(3t^2/6)\sum_j\kappa D_j^2=-16t^2=-4\pi^2a^4$.
}

The volume can also be computed directly, without using intersection numbers.
The unit $\mathbb{RP}^3$ has area $\pi^2$, and the volume defect of one tube is
\[
 -\pi^2\beta_a\int_0^c r^3\psi(r)\,\dd r=-\frac{\pi^2a^4}{4}.
\]
Extend $\psi_j$ by zero outside the $j$th tube. This gives
\begin{align}
 \dd V_a:=\frac{(\omega_a^{\mathrm{bg}})^3}{6}&=h_a\mu,
 &h_a&=1-\beta_a\sum_{j=1}^{16}\psi_j,\nonumber\\
 V_a:=\int_X\dd V_a&=\frac{1-16\pi^2a^4}{4},
 &c_a&=\frac{V_a}{\int_X\mu}=1-16\pi^2a^4.
 \label{geom:volume-density}
\end{align}
The exact source for the global volume equation is
\begin{equation}
 \mathcal R_a=\frac{c_a}{h_a}-1,\qquad
 \int_X\mathcal R_a\,\dd V_a=0.
 \label{geom:source}
\end{equation}
Outside the tubes and on their flat collars, $\mathcal R_a=-16\pi^2a^4$, rather than zero.
The final correction must therefore be allowed to change these flat regions.
The Ricci form of the background is supported in the bump annuli.
With the radii fixed as above, its annular errors in every fixed derivative order are $O(a^4)$.

The exact source and the fixed flat collars are the inputs to the
linear construction. We now resolve that operator into explicitly known
angular, tangential, and radial modes.

\section{Spherical--Fourier modes and local Dirichlet kernels}
\label{kernel:local-section}

The product geometry separates the background operator into spherical
harmonics, tangential Fourier modes, and one-dimensional radial equations.
We first specify the mode bases, then construct the regular radial branch
and the Dirichlet kernel with quantitative initialization and transfer errors.

\subsection{Moment coordinates and spherical--Fourier bases}
Define
\[
 x=D(\rho)^{1/4},\qquad
 F(x)=\frac{\rho^2[1-\beta_a\psi(\sqrt\rho)]}{D(\rho)},
\]
where $\rho$ on the right is determined by the strictly increasing function $D$.
A direct change of variables gives the real Riemannian metric and volume form
\begin{align}
 g_a&=F^{-1}\dd x^2+x^2g_H+x^2F\eta^2+|\dd w|^2,\nonumber\\
 \dd V_a&=x^3\dd x\,\dd\mathrm{vol}_{\mathbb{RP}^3}\,
                  \dd\mathrm{vol}_E.
 \label{kernel:moment-metric}
\end{align}
Here $g_H+\eta^2$ is the Hopf decomposition of the unit $S^3$, descending to $\mathbb{RP}^3$.
On the Eguchi--Hanson core, $F=1-a^4/x^4$.
On $x\in[c,b]$, one has $x=r$ and $F=1$.
We use $\Delta_{g_a}=\operatorname{div}_{g_a}\nabla$, with nonpositive
spectrum on the compact manifold. Our linearized-operator convention is
\begin{equation}
 L_a u=\Lambda_{\omega_a^{\mathrm{bg}}}(\ii\partial\bar\partial u)
       =\tfrac12\Delta_{g_a}u.
 \label{kernel:laplacian-convention}
\end{equation}
Choose a spherical basis satisfying
\[
 \Delta_{S^3}Y_{\ell,m,s}=-\ell(\ell+2)Y_{\ell,m,s},\qquad
 TY_{\ell,m,s}=\ii mY_{\ell,m,s},
\]
where $T$ is the Hopf vector field and
\[
 \ell\in2\mathbb Z_{\ge0},\qquad
 m=-\ell,-\ell+2,\ldots,\ell,\qquad s=1,\ldots,\ell+1.
\]
The basis is determined by finite polynomial algebra. Start from normal
monomials of fixed bidegree, apply the harmonic projection below, remove
linear dependencies in a fixed order, and orthonormalize using the stated
moments. The bidegree fixes the Hopf eigenvalue $m$; only even total degrees
descend to the normal quotient.
{
For definiteness, order the exponent tuples
$(\alpha_1,\alpha_2,\gamma_1,\gamma_2)$ of $z^\alpha\bar z^\gamma$
lexicographically. Apply ordered Gram--Schmidt to the harmonic
projections, discard zero-norm residuals, and normalize each retained
residual by the positive square root of its squared moment norm.
}
For a homogeneous polynomial $P$ of total degree $\ell$, its harmonic projection is
\begin{equation}
 \mathcal P_\ell P=
 \sum_{j=0}^{\lfloor\ell/2\rfloor}
 \frac{(-1)^j(\ell-j)!}{4^j j!\ell!}
 |z|^{2j}\Delta_{\mathbb R^4}^{\,j}P.
 \label{kernel:harmonic-projection}
\end{equation}
For monomials of even total degree, which descend to the quotient,
the required moments are
\begin{equation}
 \int_{\mathbb{RP}^3}z^\alpha\bar z^\gamma\,\dd S
 =\begin{cases}
 \displaystyle\frac{\pi^2\alpha_1!\alpha_2!}{(|\alpha|+1)!},&\alpha=\gamma,\\
 0,&\alpha\ne\gamma.
 \end{cases}
 \label{kernel:sphere-moments}
\end{equation}
{ For $Q$ homogeneous of degree $\ell-2j$,
$\Delta(r^{2j}Q)=4j(\ell-j+1)r^{2j-2}Q+r^{2j}\Delta Q$.
The projection coefficients satisfy
$A_{j-1}+4j(\ell-j+1)A_j=0$, so adjacent terms cancel.
Coordinate rotations force $\alpha=\gamma$ in the moment integral;
then $t=|z_1|^2$ gives the beta integral
$\alpha_1!\alpha_2!/(|\alpha|+1)!$, and the degree-two quotient
halves the sphere factor to $\pi^2$.
}

Along $E$, use
\[
 e_n(w)=\exp\bigl(2\pi\ii[n_1\Re w+n_2\Im w]\bigr),
 \quad n\in\mathbb Z^2,\quad k=2\pi|n|.
\]
For each source mode, the radial equation is
\begin{equation}
 (x^3Fv')'-P_{\ell,m,k}(x)v=2x^3f_{\ell,m,s,n}(x),
 \quad
 P_{\ell,m,k}=x\left[\ell(\ell+2)-m^2+\frac{m^2}{F}\right]+k^2x^3\ge0.
 \label{kernel:mode-ode}
\end{equation}
In particular, $\ell=0$ forces $m=0$, but the tangential frequency may still be nonzero.

\subsection{The regular core branch and Volterra propagation}
Set
\[
 L_s=\ell(\ell+2),\quad y=x^2/a^2,\quad t=(y-1)/2,
 \quad\sigma=|m|/4,\quad d=k^2a^2/4.
\]
On the Eguchi--Hanson core, the homogeneous equation becomes
\begin{equation}
 [t(1+t)v_t]'-\left[\frac{L_s}{4}
 +\frac{\sigma^2}{t(1+t)}+d(1+2t)\right]v=0.
 \label{kernel:eh-ode}
\end{equation}
Its regular branch is $v=t^\sigma\sum_{j\ge0}c_jt^j$, with
$c_0=1$ and $c_{-1}=c_{-2}=0$.
Put
\[
 C_0=2\sigma^2+\sigma-L_s/4-d,\qquad
 C_1=\sigma(\sigma+1)-L_s/4-3d,\qquad C_2=-2d.
\]
The complete recurrence is
\begin{equation}
 c_{j+1}=-\frac{
 [2j^2+(4\sigma+1)j+C_0]c_j
 +[(j-1)(j+2\sigma)+C_1]c_{j-1}+C_2c_{j-2}}
 {(j+1)(j+2\sigma+1)}.
 \label{kernel:frobenius}
\end{equation}
The nearest singularity of the coefficients is $t=-1$, so this regular series converges for $|t|<1$.
For $m=0$, take the branch without a logarithm.
For $m\ne0$, the exponent agrees with smooth normal polar coordinates on the resolution.
Indeed, $t$ is a smooth positive multiple of $|Q|^2$ near the zero
section, and the fiber angle is twice the Hopf angle; since $m$ is even,
the radial factor $t^{|m|/4}$ has the required vanishing order.
{Indeed, $x^4=a^4+|Q|^2(1+|\zeta|^2)^2$ makes $t/|Q|^2$
smooth and positive at the divisor. In a bundle trivialization,
$t^{|m|/4}e^{\ii m\theta}$ is a smooth factor times $Q^{m/2}$
for $m\ge0$, or $\overline Q^{-m/2}$ for $m<0$.
The integer powers and smooth base-section factor establish
smoothness of the complete mode.
}
The series supplies local initial data; it is not used beyond its convergence region.
For an explicit initialization bound, define
\[
 B_c=1+(4\sigma+3+|C_0|)+(1+2\sigma+|C_1|)+|C_2|.
\]
The recurrence gives $|c_j|\le B_c^j$ by induction.
{After division by $(j+1)(j+2\sigma+1)$, the three numerator
coefficients have absolute bounds
$4\sigma+3+|C_0|$, $1+2\sigma+|C_1|$, and $|C_2|$.
Their sum is $B_c-1$; negative-index terms vanish.
Induction gives the bound, and the following tails are geometric
series and their derivatives.
}
If the core boundary corresponds to $t_{\rm core}>0$, take
$t_0=\min(1/(4B_c),t_{\rm core}/2)$.
Writing $h(t)=\sum_{j\ge0}c_jt^j$, the value and derivative tails at this point satisfy
\begin{align}
 \left|h(t_0)-\sum_{j=0}^N c_jt_0^j\right|
 &\le\frac43\,4^{-(N+1)},\nonumber\\
 \left|h'(t_0)-\sum_{j=1}^N jc_jt_0^{j-1}\right|
 &\le B_c\sum_{j=N+1}^{\infty}j4^{-(j-1)}
 =B_c4^{-N}\left(\frac{4(N+1)}3+\frac49\right).
 \label{kernel:frobenius-tail}
\end{align}
These bounds give explicit initial errors for $v=t^\sigma h$,
$v_t=\sigma t^{\sigma-1}h+t^\sigma h'$, and
$J_{\rm flux}=x^4Fv_t/a^2$.

On a compact interval away from the bolt, let $J_{\rm flux}=x^3Fv'$.
The homogeneous first-order system is
\[
 \binom v{J_{\rm flux}}'
 =\mathcal A(x)\binom v{J_{\rm flux}},\qquad
 \mathcal A(x)=\begin{pmatrix}0&(x^3F)^{-1}\\P_{\ell,m,k}&0\end{pmatrix}.
\]
Its transfer matrix is given by the explicit ordered-integral series
{with $x\ge x_0$:}
\begin{equation}
 \mathcal T(x,x_0)=\Id+
 \sum_{N\ge1}\int_{x_0<t_1<\cdots<t_N<x}
 \mathcal A(t_N)\cdots\mathcal A(t_1)\,\dd t_1\cdots\dd t_N.
 \label{kernel:volterra}
\end{equation}
If $B=\int_{x_0}^x\|\mathcal A(t)\|\,\dd t$, the $N$th term has norm at most $B^N/N!$.
The remainder after truncation at order $N$ is at most $e^BB^{N+1}/(N+1)!$.
The inhomogeneous term is included by variation of constants.
This construction crosses the smooth transition annulus without requiring real-analytic coefficients there.
For each fixed mode, the transfer series converges and takes neither an unknown ODE solution nor an unknown global Green kernel as input.

\subsection{Local Green kernels and a zeroth-order bound}
Since $P\ge0$ and the regular flux vanishes at the bolt, the regular homogeneous solution $v_{\rm reg}$ is positive and nondecreasing.
{Positivity here means $x>a$. The initial Frobenius flux is
$O(t^\sigma)$ for $\sigma>0$ and $O(t)$ otherwise, hence tends to zero.
While $v_{\rm reg}>0$,
$x^3Fv_{\rm reg}'=\int_a^x Pv_{\rm reg}\,\dd y\ge0$;
thus the branch cannot acquire a first zero.
}
Relative to the radial measure $y^3\dd y$, the zero-Dirichlet kernel at $x=b$ is
\begin{equation}
 K_{\ell,m,k}(x,y)=
 -2v_{\rm reg}(x)v_{\rm reg}(y)
 \int_{\max(x,y)}^b\frac{\dd t}{t^3F(t)v_{\rm reg}(t)^2}.
 \label{kernel:radial-green}
\end{equation}
Its weighted radial flux jump is
\[
 y^3F(y)\bigl[\partial_xK(y^+,y)-\partial_xK(y^-,y)\bigr]=2,
\]
in agreement with~\eqref{kernel:laplacian-convention}.
For the constant mode, $v_{\rm reg}=1$, giving $K_{0,0,0}$, and
\[
 0\le-K_{\ell,m,k}(x,y)\le-K_{0,0,0}(x,y).
\]
{For a projected smooth source, circle-weight Taylor expansion gives
$f=t^\sigma g(t)$ with $g$ smooth. Writing
$v=v_{\rm reg}=t^\sigma h(t)$, $h(0)=1$, the Green integral satisfies
\[
 (u/v)'(x)=
 \frac{2\int_a^x v(y)y^3f(y)\,\dd y}{x^3F(x)v(x)^2}.
\]
Both numerator and denominator are $t^{2\sigma+1}$ times smooth
functions, with nonzero denominator factor.
The endpoint integral for $u/v$ is finite: its potentially singular
integrand is bounded for $\sigma>0$ and at worst logarithmic otherwise.
Hence $u=t^\sigma\widetilde h(t)$ with $\widetilde h$ smooth;
the complete mode extends across the bolt.
}
The logarithmic singularity of the zero-mode kernel at the bolt is locally
integrable against $x^3\dd x$. The mode kernels therefore define the
Dirichlet inverse on smooth data. To justify their summation, use the
compact $U(2)\times E$ isometry action of the product background.
Its infinitesimal generators extend smoothly across the exceptional
divisor. Applying powers of the spherical and tangential Casimir operators
to a smooth source gives arbitrarily rapid decay of its mode coefficients,
uniformly in the radial variable for fixed $a>0$.
The addition formula bounds the spherical basis values by a polynomial in
$\ell$, while the Fourier basis has modulus one.
{Precisely, $\|Y_{\ell,m,s}\|_\infty\le(\ell+1)/\pi$ and the
degree-$\ell$ multiplicity is $(\ell+1)^2$. Casimir decay reduces
uniform summation to
\[
 \sum_{\ell\ {\rm even}}(\ell+1)^3
 \sum_{n\in\mathbb Z^2}
 (1+\ell(\ell+2)+4\pi^2|n|^2)^{-N}<\infty
 \qquad(N>3).
\]
}
The comparison with $-K_{0,0,0}$ then gives an integrable majorant for the
resulting potential series, and the zero-mode barrier below bounds its
radial integral. Thus the sum converges uniformly on smooth sources.
Equivalently, insert heat factors in the known compact-group modes and let
them tend to one. Testing the modal equations against smooth functions
shows that the limit solves the Dirichlet equation weakly. The local
background is smooth on the resolution, so elliptic regularity yields the
smooth zero-boundary solution. This also justifies the harmonic extension
multipliers used in the global construction.
{The finite sums are smooth at the bolt; uniform convergence of
sources and potentials passes the tested equations and zero boundary
values to the limit. For harmonic extension,
$v_{\rm reg}(x)/v_{\rm reg}(b)\le1$ gives the same argument.
All constants in this summation may depend on the fixed $a$.
}

For fixed $a$, let $u_N$ be the finite zero-boundary modal sums on
$U_j$, with $f_N=L_au_N$. Integration by parts gives
\[
 \|\nabla(u_N-u_K)\|_{L^2(U_j,\dd V_a)}^2
 =-2\operatorname{Re}\int_{U_j}
       \overline{u_N-u_K}(f_N-f_K)\,\dd V_a
 \le2\Vol_{g_a}(U_j)\|u_N-u_K\|_\infty
                         \|f_N-f_K\|_\infty\longrightarrow0.
\]
Uniform convergence also gives $L^2$ convergence, so the limit lies
in $H^1_0(U_j)$ and is the weak Dirichlet solution. Smooth boundary
elliptic regularity gives smoothness up to the artificial boundary.
For smooth harmonic boundary data, a smooth extension followed by
this zero-boundary inverse gives a smooth harmonic solution; the
maximum principle identifies it with the uniformly convergent modal
extension.

The identity $D=\rho^2+a^4-4\beta_a J$ implies
\[
 F\ge(1-\beta_a)(1-a^4/x^4).
\]
The radial zero-boundary barrier solving $L_au=-1$ satisfies
\[
 u'(x)=-\frac{x^4-a^4}{2x^3F(x)},\qquad
 0\le u\le\frac{b^2-a^2}{4(1-\beta_a)}.
\]
The maximum principle, together with~\eqref{geom:radii} and~\eqref{geom:bump-bounds}, therefore gives
\begin{equation}
 \|L_{a,D}^{-1}f\|_{L^\infty}
 \le\frac{b^2}{4(1-\beta_a)}\|f\|_{L^\infty}
 \le\frac{65}{65536}\|f\|_{L^\infty}.
 \label{kernel:local-sup-bound}
\end{equation}
The local kernels determine the source solutions and harmonic extensions
on each resolved tube. The fixed collar now controls their difference
from the corresponding flat-tube extensions.

\section{The collar estimate and global transmission inverse}
\label{kernel:global-section}
The local operators and the flat quotient kernel determine the global
inverse by matching boundary traces and normal derivatives. We first prove
the relative collar bound, then construct the single-layer inverse in its
compatible energy space, and finally establish global regularity and mean
normalization. The boundary contraction is an energy estimate; the
weighted inverse norm is proved in the next section.

\subsection{The Schoen collar estimate at all tangential frequencies}
Let $N_a$ be the direct sum of the Dirichlet-to-Neumann operators of the sixteen resolved tubes, with normals pointing from the tubes into the exterior.
Let $N_0$ be the corresponding operator for the flat tubes with the same boundaries, and put $D_N=N_a-N_0$.
Both metrics are flat on $[c,b]$.
The regular flat radial solution $f$ satisfies
\[
 f''+3f'/x-[\ell(\ell+2)/x^2+k^2]f=0.
\]
For $k>0$, one may take $f=x^{-1}I_{\ell+1}(kx)$; for $k=0$, take $f=x^\ell$.
Write $\lambda_0=f'(b)/f(b)$ and
$\delta=v_{\rm reg}'(b)/v_{\rm reg}(b)-\lambda_0$.
{For $\nu=\ell+1$, the series
$I_\nu(z)=\sum_{j\ge0}(z/2)^{2j+\nu}/(j!\,\Gamma(j+\nu+1))$
satisfies $z^2I_\nu''+zI_\nu'-(z^2+\nu^2)I_\nu=0$
by coefficient comparison. Substitution verifies the flat equation
and the positive regular branch; $k=0$ gives $x^\ell$ directly.
Only the complete constant mode has $\lambda_0=0$.
}

\begin{proposition}
\label{kernel:dtn-bound}
{
Here boundary data means $h\in H^{1/2}(\Sigma)$, where $\Sigma$ is
the union of the sixteen common tube boundaries; the brackets denote
the Hermitian $H^{1/2}$--$H^{-1/2}$ duality pairing.
}
For all boundary data,
\begin{equation}
 |\langle h,D_Nh\rangle|
 \le\frac1{15}\langle h,N_0h\rangle.
 \label{kernel:dtn-form-bound}
\end{equation}
The operator $D_N$ vanishes on each local constant mode, so the bound also accounts for all sixteen zero modes.
\end{proposition}

\begin{proof}
The tangential Fourier index is $n\in\mathbb Z^2$, with
$\kappa=2\pi|n|$. The spherical harmonic degree $\ell$ on
$\mathbb{RP}^3$ is a nonnegative even integer. The angular Hopf index satisfies
$|m|\le\ell$, so $\ell=0$ forces $m=0$.
The core modal potential
\[
 P_{\ell,m,\kappa}
 =x\bigl[\ell(\ell+2)-m^2+m^2/F_a\bigr]+\kappa^2x^3
\]
is nonnegative. Thus the regular solution $v$ stays positive and is
nondecreasing. Denote the regular flat solution by $v_0$; it satisfies
\begin{equation}\label{c0:flat-mode}
 v_0''+\frac3r v_0'
 -\left(\frac{\ell(\ell+2)}{r^2}+\kappa^2\right)v_0=0.
\end{equation}
We may take $v_0=r^{-1}I_{\ell+1}(\kappa r)$ when $\kappa>0$, and
$v_0=r^\ell$ when $\kappa=0$. Write
$\lambda_0=v_0'(b)/v_0(b)$ and $\delta=v'(b)/v(b)-\lambda_0$.

For $\ell=0$ and $\kappa>0$, the product volume form and the zero endpoint
flux give
\[
 c^3v'(c)=\kappa^2\int_a^c x^3v(x)\,\dd x,
 \qquad
 0\le\frac{v'(c)}{v(c)}\le\frac{\kappa^2c}{4}.
\]
The logarithmic derivative of the flat solution at $c$ lies in the same
interval $[0,\kappa^2c/4]$. Constancy of the Wronskian on the collar yields
\[
 \delta=\left(\frac cb\right)^3
 \frac{v(c)}{v(b)}\frac{v_0(c)}{v_0(b)}
 \left(\frac{v'(c)}{v(c)}-\frac{v_0'(c)}{v_0(c)}\right).
\]
The difference of the two entrance logarithmic derivatives is at most
$\kappa^2c/4$, rather than twice this quantity. Using monotonicity of $v$ and
$\lambda_0=\kappa^2\int_0^b r^3v_0(r)\,\dd r/[b^3v_0(b)]$, we obtain
\begin{equation}\label{c0:mass-zero-angular}
 \frac{|\delta|}{\lambda_0}
 \le\frac{c^4v_0(c)}{4\int_0^b r^3v_0(r)\,\dd r}
 \le\frac{c^4}{b^4-c^4}=\frac1{15}.
\end{equation}
This covers arbitrarily large $\kappa$. The actual constant mode
$\ell=\kappa=0$ has $\delta=0$.

To treat the remaining modes and quantify high-frequency smoothing, set
$\alpha=rv_0'/v_0$. For $\kappa>0$, put $z=\kappa r$. The Riccati equation is
\begin{equation}\label{c0:riccati}
 z\alpha_z=\ell(\ell+2)+z^2-\alpha(\alpha+2).
\end{equation}
The regular expansion $\alpha=\ell+z^2/(2\ell+4)+O(z^4)$ and a first-crossing
comparison with $-1+\sqrt{(\ell+1)^2+z^2}$ imply
\begin{equation}\label{c0:alpha-bounds}
 \alpha\text{ is nondecreasing},\qquad
 \ell\le\alpha\le\ell+z,\qquad
 \lambda_0\le16\ell+\kappa.
\end{equation}
More explicitly, the expansion starts below this comparison function.
At a first contact the Riccati equation would give $\alpha_z=0$,
while the comparison function has positive derivative. Thus no contact
occurs, the right-hand side of \eqref{c0:riccati} is positive, and
$\alpha$ is nondecreasing.
When $\kappa=0$, one has $\alpha=\ell$, so the same conclusions hold.
The decreasing branch can be written as
\[
 g(r)=v_0(r)J_0(r),\qquad
 J_0(r)=\int_r^\infty\frac{\dd t}{t^3v_0(t)^2}.
\]
Monotonicity of $\alpha$ gives
$J_0(r)\le[(2\alpha(r)+2)r^2v_0(r)^2]^{-1}$, and hence
$\beta_0:=-rg'/g\ge\alpha+2>\alpha$.
Let $d_c=v'(c)/v(c)\ge0$. If $v=A v_0+B g$, then
\[
 \tau:=\frac{Bg(c)}{Av_0(c)}
 =\frac{\alpha(c)-cd_c}{\beta_0(c)+cd_c}
 \in\left(-1,\frac{\alpha(c)}{\beta_0(c)}\right],
 \qquad
 \eta=\frac{J_0(b)}{J_0(c)}\in(0,1).
\]
The exact identity
$\delta=-(\alpha(b)+\beta_0(b))\tau\eta/[b(1+\tau\eta)]$ implies the
Dirichlet entrance-limit bound for the full normal derivative discrepancy:
\begin{equation}\label{c0:dirichlet-extreme}
 |\delta|\le\delta_D
 :=\frac1{b^3v_0(b)^2\displaystyle\int_c^b
                     \frac{\dd t}{t^3v_0(t)^2}}.
\end{equation}
{Here $Av_0(c)=v(c)(\beta_0(c)+cd_c)/(\alpha(c)+\beta_0(c))>0$.
For negative $\tau$, $|\tau|/(1+\tau\eta)\le(1-\eta)^{-1}$;
for nonnegative $\tau<1$, it is at most $(1+\eta)^{-1}$.
Finally,
$\alpha(b)+\beta_0(b)=[b^2v_0(b)^2J_0(b)]^{-1}$,
and $J_0(c)-J_0(b)=\int_c^b t^{-3}v_0(t)^{-2}\,\dd t$.
These identities give the bound for both signs of $\delta$.
}
For $\ell\ge2$, use $v_0(t)/v_0(b)\le(t/b)^\ell$ and
$\alpha(b)\ge\ell$ to obtain
\begin{equation}\label{c0:angular-q}
 \frac{|\delta|}{\lambda_0}\le q_\ell
 :=\left(2+\frac2\ell\right)
   \frac{4^{-\ell-1}}{1-4^{-\ell-1}}
 \le\frac{16}{21}4^{-\ell},\qquad q_2=\frac1{21}.
\end{equation}
Together with \eqref{c0:mass-zero-angular} and the constant mode, this gives
\begin{equation}\label{c0:relative-DN}
 |\langle h,D_Nh\rangle|
 \le\frac1{15}\langle h,N_0h\rangle.
\end{equation}
The operator $D_N$ annihilates all sixteen local constant modes.

\end{proof}

\subsection{The affine quotient heat kernel}
For $L_0=\Delta_{\mathbb R^6}/2$ on the flat torus $A$, the heat kernel is
\[
 H_A(t;x,y)=\sum_{\xi\in\mathbb Z^6}
 e^{-2\pi^2|\xi|^2t}e^{2\pi\ii\xi\cdot(x-y)}.
\]
The zero-mean Green kernel on $Y$ is the explicit integral
\begin{equation}
 G_Y(x,y)=-\int_0^\infty
 \left[\sum_{\gamma\in G}H_A(t;x,\gamma y)-4\right]\,\dd t.
 \label{kernel:flat-green}
\end{equation}
The group action includes the affine translations specified above.
The quotient kernel uses a group sum, not a group average.
\subsection{Domains, normals, and the collar hypothesis}
\label{kernel:transmission-setup}

Let $B_j=E\times(B^4_b/\{\pm1\})$, $1\le j\le16$, be the
artificial flat tubes in $Y$, and let $M_j=U_j$ be their resolved
counterparts. Their common boundaries are $\Sigma_j$. The exterior
$\mathcal E$ is connected: paths can be perturbed away from the singular
curves and rerouted along the connected boundaries of the excised tubes.
The background is flat on the exterior and on
an open collar inside every resolved tube. Its real metric is denoted by
$g_a$, and $L_a=\Delta_{g_a}/2$. Write
\[
 \Sigma=\coprod_{j=1}^{16}\Sigma_j,\qquad V_Y=\tfrac14,
 \qquad V_a=\int_X\dd V_a.
\]
All normals \(n_j\) point from the interiors into \(\mathcal E\).
Thus the exterior outward normal is \(-n_j\). Boundary integrals use
the common induced flat measure \(dS\).

The local constructions supply explicit zero-boundary Dirichlet solution
operators and harmonic extension operators \(H_{a,j}\) on \(M_j\).
Let \(N_a\) be the direct sum of their harmonic Dirichlet-to-Neumann
maps, and let \(N_0\) be the corresponding map on the artificial flat
domains \(B_j\). Harmonic functions on the flat quotient domains are
understood in the regular orbifold, equivalently Friedrichs energy, sense.
Both maps are selfadjoint, nonnegative, and annihilate a constant on
each individual boundary component. Set
\(D_N=N_a-N_0\). The collar estimate \eqref{kernel:dtn-form-bound} holds with $q=1/15$.
The pairings below use the usual Hermitian convention when the data are
complex-valued.

The common collar also implies that \(D_N\) is smoothing. Indeed, the
difference of the two harmonic extensions of the same boundary datum is
harmonic in that collar and has zero trace on \(\Sigma\). Interior and
homogeneous Dirichlet boundary estimates, applied on a smaller collar,
bound every derivative of its normal trace by the initial \(H^1\) norm.
Consequently \(D_N:H^{1/2}(\Sigma)\to H^m(\Sigma)\) is bounded for every
fixed \(m\), for each fixed \(a>0\). This observation requires no
quantitative estimate for the geometry away from the collar.

\subsection{Single layers and the compatible energy space}
\label{kernel:single-layers}

Let \(G_Y\) be the specified zero-mean flat Green kernel, obtained from
the finite-group sum of the torus heat kernel. With
\(L_0=\Delta_{g_0}/2\), its normalization is
\[
 L_{0,x}G_Y(x,y)=\delta_y-V_Y^{-1},\qquad
 \int_YG_Y(x,y)\,dV_0(x)=0.
\]
For a boundary density \(\sigma\), define
\begin{equation}\label{kernel:single-layer}
 (\mathcal S\sigma)(x)=-\frac12\int_\Sigma G_Y(x,y)\sigma(y)\,dS_y,
 \qquad S\sigma=(\mathcal S\sigma)|_\Sigma.
\end{equation}
For densities in $H^{-1/2}(\Sigma)$, the potential has matching
$H^{1/2}$ traces across $\Sigma$; it is continuous for smooth densities.
Distributionally,
\begin{equation}\label{kernel:layer-jump}
 \begin{aligned}
 L_0\mathcal S\sigma
   &=-\tfrac12\sigma\delta_\Sigma
      +\frac1{2V_Y}\int_\Sigma\sigma\,dS,\\
 \partial_n^{\mathrm{ext}}\mathcal S\sigma
       -\partial_n^{\mathrm{int}}\mathcal S\sigma&=-\sigma.
 \end{aligned}
\end{equation}
The symmetry of $G_Y$ gives selfadjointness of the boundary operator $S$.
The factor \(-1/2\) accounts for the convention \(L_0=\Delta/2\).
The constant source in \eqref{kernel:layer-jump} must be retained when
the density has nonzero total integral.
{The trace theorem places $\sigma\delta_\Sigma$ in $H^{-1}$.
After constant subtraction, Fourier inversion on the invariant torus
lift gives $H^1$ regularity and matching $H^{1/2}$ traces.
For smooth densities, the surface term in $\Delta U$ is the displayed
normal-derivative jump; $\Delta=2L_0$ accounts for $-1/2$.
Testing the weak equation with the layer gives the energy identity.
}

For the proof, let \(N_{\mathrm{out}}\) be the harmonic
Dirichlet-to-Neumann map on \(\mathcal E\), using its outward normal
\(-n\), and put \(A_0=N_0+N_{\mathrm{out}}\). Its quadratic form is
the sum of the interior and exterior harmonic energies. Its kernel is
one common constant: zero energy makes every harmonic extension
constant, and continuity and connectedness force all those constants
to agree. The appropriate spaces are therefore
\[
 \mathscr D_0=\left\{\sigma\in H^{-1/2}(\Sigma):
                    \langle\sigma,1\rangle_\Sigma=0\right\},
 \qquad \mathscr H_0=H^{1/2}(\Sigma)/\mathbb C.
\]
For \(\sigma\in\mathscr D_0\), the layer is harmonic on both sides and
the jump relation gives \(A_0S\sigma=\sigma\). Furthermore,
\begin{equation}\label{kernel:energy-norms}
 \begin{aligned}
 \|\sigma\|_S^2
    &:=\langle\sigma,S\sigma\rangle_\Sigma
      =\int_Y|\nabla\mathcal S\sigma|^2\,dV_0,\\
 \|h\|_{E_0}^2&:=\langle h,A_0h\rangle_\Sigma.
 \end{aligned}
\end{equation}
{Glue the harmonic extensions into $U_h\in H^1(Y)$.
Poincar\'e on the invariant torus lift and the trace theorem give
$\|[h]\|_{H^{1/2}/\mathbb C}\le C\|\nabla U_h\|_2$;
bounded trace extensions and energy minimization give the reverse
bound. Riesz representation yields the isomorphism and completeness;
$A_0S=\Id$ identifies the inverse and the dual density norm.
}
These are positive definite on \(\mathscr D_0\) and \(\mathscr H_0\),
respectively. The elliptic operator \(A_0\) induces an isomorphism
\(\mathscr H_0\to\mathscr D_0\); its inverse is the class of \(S\).
The energy norms are equivalent to the corresponding Sobolev norms
on these spaces, and hence are complete. The exterior map and this
abstract inverse are used only to prove the estimates: the construction
evaluates \(S\) through \eqref{kernel:single-layer}.
{On the free boundary quotient, the Dirichlet-to-Neumann symbol theorem applies in the smooth product lifts. Each outward map has symbol
$|\xi|$ (the inward harmonic mode is $e^{-t|\xi|}$), so
$A_0$ has symbol $2|\xi|>0$. These fixed-geometry norm equivalences
supply no numerical inverse constant.
}

\begin{proposition}[The boundary transmission inverse]
\label{kernel:boundary-inverse}
Under \eqref{kernel:dtn-form-bound}, \(D_NS\) is selfadjoint in the
\(S\)-energy inner product on \(\mathscr D_0\), and
\(\|D_NS\|_S\le q\). Consequently
\begin{equation}\label{kernel:boundary-neumann}
 (I+D_NS)^{-1}g=\sum_{n\ge0}(-D_NS)^ng,
 \qquad
 \left\|\sum_{n>N}(-D_NS)^ng\right\|_S
 \le\frac{q^{N+1}}{1-q}\|g\|_S
 \quad(N\ge0).
\end{equation}
\end{proposition}
\begin{proof}
Selfadjointness of \(D_N\) and \(S\) gives selfadjointness of \(D_NS\)
in the indicated inner product. Since exterior harmonic energy is
nonnegative,
\[
 |\langle S\sigma,D_NS\sigma\rangle_\Sigma|
 \le q\langle S\sigma,N_0S\sigma\rangle_\Sigma
 \le q\langle S\sigma,A_0S\sigma\rangle_\Sigma
 =q\|\sigma\|_S^2.
\]
The selfadjoint operator norm is therefore at most \(q\), which proves
the convergent series and its tail estimate. Since \(D_N\) annihilates
the componentwise constants and is selfadjoint, its output has zero
integral on every component. Thus \(D_NS\) preserves compatibility.
\end{proof}

At $q=1/15$, the density tail in \eqref{kernel:boundary-neumann} is
\begin{equation}\label{kernel:boundary-tail}
 \left\|\sum_{n>N}(-D_NS)^ng\right\|_S
 \le\frac{15}{14}\,15^{-(N+1)}\|g\|_S.
\end{equation}
\subsection{The global inverse, regularity, and mean normalization}
\label{kernel:global-construction}

Let \(f\in C^\infty(X)\) satisfy \(\int_Xf\,dV_a=0\). The explicit
local Dirichlet kernels determine \(w_j\) with
\(L_aw_j=f|_{M_j}\) and \(w_j|_{\Sigma_j}=0\). Set
\(j_j=\partial_nw_j\), and write \(j=(j_j)_j\). Define
\begin{equation}\label{kernel:global-input}
 \begin{aligned}
 v_0(x)&=\int_{\mathcal E}G_Y(x,y)f(y)\,dV_0(y),\\
 W&=v_0-\mathcal Sj,\qquad h_0=W|_\Sigma,\\
 \sigma&=\sum_{n\ge0}(-D_NS)^nD_Nh_0,\qquad h=h_0-S\sigma.
 \end{aligned}
\end{equation}
{Initially, Fourier inversion gives an $H^2$ volume potential and
the trace argument gives an $H^1$ layer. In a smooth collar,
after absorbing the volume density,
\[
 \partial_i(A^{ij}\partial_jU)
 =F_+1_{\{t>0\}}+F_-1_{\{t<0\}}+G(s)\delta_{\{t=0\}},
\]
with smooth elliptic $A$, smooth one-sided $F_\pm$ and smooth $G$.
Subtracting $V=\chi(t)t_+G(s)/A^{tt}(s,0)$, with $\chi=1$ near zero,
cancels the conormal jump.
}
Although \(f1_{\mathcal E}\) is only piecewise smooth, the volume
potential solves a local transmission problem with smooth right-hand
sides on both sides and continuous trace and normal derivative. Smooth
interface regularity gives smooth one-sided restrictions up to
\(\Sigma\). The single layer with smooth density \(j\) has the same
one-sided regularity, with its prescribed smooth derivative jump. Thus
\(h_0\in C^\infty(\Sigma)\). The traces in
\eqref{kernel:global-input} are actual functions; the quotient by a
common constant is used only for their energy estimates.
{ For $Z=U-V$, tangential difference quotients satisfy
\[
 \partial_i(A^{ij}\partial_jD_hZ)
 =D_hF-\partial_i\bigl((D_hA^{ij})
                 (\partial_jZ)(s+he_\alpha,t)\bigr),
\]
where $F$ is the resulting piecewise smooth forcing.
A cutoff energy estimate, using
$\|D_hZ\|_2\le\|\partial_\alpha Z\|_2$, gives an $h$-uniform local
$H^1$ bound. Iteration controls every tangential derivative.
Since $A^{tt}>0$, the equation recovers $Z_{tt}$ on each side;
differentiation inductively recovers all normal derivatives.
Sobolev embedding gives smooth one-sided restrictions and a smooth
common trace, for both the volume potential ($G=0$) and the layer.
}

Set
\begin{equation}\label{kernel:global-representation}
 \begin{aligned}
 \widetilde u|_{\mathcal E}&=W-\mathcal S\sigma,\\
 \widetilde u|_{M_j}&=w_j+H_{a,j}h_j,\\
 G_af&=\widetilde u-\frac1{V_a}\int_X\widetilde u\,dV_a.
 \end{aligned}
\end{equation}
\begin{theorem}[An explicit global zero-mean inverse]
\label{kernel:global-inverse}
For the fixed special Schoen background, formulas
\eqref{kernel:global-input}--\eqref{kernel:global-representation} determine
the unique smooth solution of
\[
 L_aG_af=f,\qquad \int_XG_af\,dV_a=0.
\]
Its boundary density is determined by the convergent series
\eqref{kernel:boundary-neumann}. Every operation uses the prescribed
local Dirichlet kernels and harmonic extensions, the flat quotient
Green kernel, boundary integrals, and a convergent boundary series.
\end{theorem}
\begin{proof}
The divergence theorem gives
\[
 \int_{\Sigma_j}j_j\,dS=2\int_{M_j}f\,dV_a,
 \qquad \int_\Sigma j\,dS=-2\int_{\mathcal E}f\,dV_0.
\]
Write \(I_{\mathcal E}=\int_{\mathcal E}f\,dV_0\). Then
\[
 L_0v_0=f1_{\mathcal E}-I_{\mathcal E}/V_Y,
 \qquad
 L_0\mathcal Sj=-j\delta_\Sigma/2-I_{\mathcal E}/V_Y.
\]
Thus the constant sources cancel. The function \(W\) solves
\(L_0W=f\) on \(\mathcal E\), is harmonic in every \(B_j\), and has
normal derivative
\(\partial_n^{\mathrm{ext}}W=N_0h_0+j\).

Proposition~\ref{kernel:boundary-inverse} gives
\(\sigma+D_NS\sigma=D_Nh_0\), or equivalently \(\sigma=D_Nh\).
The density \(\sigma\) has zero integral on each component, so its
layer has no constant source. Its interior normal derivative is
\(N_0S\sigma\), and \eqref{kernel:layer-jump} yields the exterior
normal derivative of \(\widetilde u\):
\[
 N_0h_0+j-N_0S\sigma+\sigma
   =N_0h+j+D_Nh=N_ah+j.
\]
This is the interior normal derivative of \(w_j+H_{a,j}h_j\).
The boundary traces agree as well, so \(\widetilde u\) is a global
weak solution with the prescribed source.

The energy series initially gives \(\sigma\in H^{-1/2}(\Sigma)\).
The equation \(\sigma=D_Nh_0-D_NS\sigma\), together with the smoothing
property of \(D_N\), improves it to \(C^\infty(\Sigma)\). The local
solutions consequently have smooth one-sided restrictions. Global
elliptic regularity for the weak solution on the smooth manifold
\(X\) gives \(\widetilde u\in C^\infty(X)\). Subtracting its mean
preserves the equation. The difference of two zero-mean solutions is
harmonic; integration of its energy makes it zero, proving uniqueness.
\end{proof}

The mean subtraction requires only local integrals. The function
\(W-\mathcal S\sigma\) has zero mean on \(Y\) and is the flat harmonic
extension \(H_{0,j}h_j\) in each artificial domain. Consequently
\begin{equation}\label{kernel:local-mean-formula}
 \int_X\widetilde u\,dV_a
 =\sum_{j=1}^{16}\left[
   \int_{M_j}w_j\,dV_a+
   \int_{M_j}H_{a,j}h_j\,dV_a-
   \int_{B_j}H_{0,j}h_j\,dV_0\right].
\end{equation}
In the product radial models, only the complete constant mode contributes
to the last two integrals. Their difference is the local volume defect
times the normalized boundary average of \(h_j\). No separate inversion
of the sixteen local constant modes is needed.

The source energy also admits a bound using specified integrals.
Harmonic extension minimizes the exterior energy, while \(W\) is already
harmonic in the artificial interiors. Since
\(\Delta_{g_0}W=2f1_{\mathcal E}+j\delta_\Sigma\), integration by parts gives
\begin{equation}\label{kernel:input-energy}
 \begin{aligned}
 \|h_0\|_{E_0}^2
 &\le\int_Y|\nabla W|^2\,dV_0\\
 &=-2\operatorname{Re}\int_{\mathcal E}\overline Wf\,dV_0
   -\operatorname{Re}\int_\Sigma\overline{h_0}j\,dS\\
 &\le2\int_{\mathcal E}|Wf|\,dV_0+
       \int_\Sigma|h_0j|\,dS.
 \end{aligned}
\end{equation}
In particular, this bound requires no exterior Dirichlet inverse.

For completeness, truncate only the boundary series at \(N\ge0\) by
\[
 \sigma_N=\sum_{n=0}^N(-D_NS)^nD_Nh_0,
 \qquad h_N=h_0-S\sigma_N.
\]
The operator \(SD_N\) is conjugate to \(D_NS\) through the energy isometry
\(\mathscr D_0\to\mathscr H_0\), and has norm at most \(q\).
The reconstructed trace retains powers zero through \(N+1\) of
\(-SD_N\). Hence
\begin{equation}\label{kernel:trace-layer-tail}
 \|h-h_N\|_{E_0}\le\frac{q^{N+2}}{1-q}\|h_0\|_{E_0}.
\end{equation}
This is a boundary energy estimate. The weighted interior estimates are
proved below. A finite spectral truncation additionally requires the
source-dependent estimates discussed in Section~\ref{computational:section}.

\section{The capped weighted inverse}\label{weighted:section}
\subsection{Weights, norms, and the mean-zero completions}
\label{conv:setup}
Throughout this subsection, $X$ is the fixed special smooth Schoen model
of Section~\ref{geom:section}. The background is
$\omega_a^{\mathrm{bg}}$, its associated real metric is $g_a$, and
$\dd V_a=(\omega_a^{\mathrm{bg}})^3/6$. We use the class, volume density,
and normalized source established in
\eqref{geom:kahler-class}--\eqref{geom:source}.
The fixed radii and cutoff are those of
\eqref{geom:radii} and \eqref{geom:bump}.

Let $s_a$ be the smooth full geometric weight of
Appendix~\ref{sch:sec:schauder}. In particular,
$s_a=D(r^2)^{1/4}$ in each normal tube up to $r=b$,
$a\le s_a\le1$, and $|\nabla s_a|\le2^{13}$.
The final nonlinear estimates use the capped weight
\begin{equation}
 \rho_a=\min\{s_a,c\},\qquad a\le\rho_a\le c.
 \label{conv:cap-weight}
\end{equation}
Thus $\rho_a=D(r^2)^{1/4}$ for $r\le c$ and $\rho_a=c$ elsewhere.
Only Lipschitz regularity is needed at the cap: the weight occurs in the
norms and is not a coefficient of the metric.

\begin{definition}[The final mean-zero norms and their completions]
\label{conv:def-norms}
{
The tensor norm and the space $E_a$ below are restricted to
$0<a\le2^{-10}$, the range of the geometric certificate in
Appendix~\ref{sch:sec:schauder}. The scalar norm and its closure
$Y_a$ require no parallel transport and remain defined for
$0<a\le1/128$.
}
Fix $\alpha=1/2$ and $\kappa_0=2^{-200}$, and write $d_a=d_{g_a}$.
For a pair $p,q$, let
$m_{pq}=\min\{\rho_a(p),\rho_a(q)\}$. The suprema below are taken over
$0<d_a(p,q)\le\kappa_0m_{pq}$.
Appendix~\ref{sch:sec:schauder} ensures uniqueness of the shortest geodesic
at this scale. Let $P_{qp}$ denote parallel
transport along that geodesic. For a scalar $f$ and a second-order tensor $B$,
define
\begin{align*}
 [f]_*&=\sup_{p,q}
  \frac{m_{pq}^{7/2}|f(p)-f(q)|}{d_a(p,q)^{1/2}},\\
 [B]_*&=\sup_{p,q}
  \frac{m_{pq}^{7/2}\|B(p)-P_{qp}B(q)\|_{\mathrm{op}}}
       {d_a(p,q)^{1/2}},\\
 \|f\|_{Y_a}&=\sup_X\rho_a^3|f|+[f]_*,\\
 \|u\|_{E_a}&=\sup_X\rho_a|u|+\sup_X\rho_a^2|\nabla u|
       +\sup_X\rho_a^3\|\nabla^2u\|_{\mathrm{op}}
       +[\nabla^2u]_*.
\end{align*}
The spaces $E_a$ and $Y_a$ are the respective closures of smooth functions
of zero $\dd V_a$ mean in these norms. Complexification uses the complex
operator norm. All covariant derivatives, distances, and tensor norms are
defined by $g_a$. For each fixed $a>0$, these norms are equivalent to
the usual $C^{2,1/2}$ and $C^{0,1/2}$ norms restricted to the corresponding
closures. Smooth functions are not assumed to be dense in the full H\"older
spaces.
\end{definition}
{ For fixed $a>0$, all pairs at distance at most $\kappa_0a$ are
included. On a finite coordinate cover, coordinate and transported
H\"older seminorms are equivalent; farther pairs are controlled
by the supremum norm. Integration is continuous, so the mean-zero
condition and the chosen smooth closures are closed.
}

Write
\begin{equation}
 A_a(u)=(\omega_a^{\mathrm{bg}})^{-1}\ii\partial\bar\partial u,
 \qquad L_au=\operatorname{tr}A_a(u)=\frac12\Delta_{g_a}u.
 \label{conv:linearization}
\end{equation}
The $L^\infty$ norm of an endomorphism below is its uniform operator norm.
The construction in Section~\ref{kernel:global-section} specifies the
mean-zero inverse $G_a$ by local
mode kernels, Frobenius recurrences, Volterra transport, the flat torus
heat kernel, and a geometric series for boundary transmission.
For smooth sources it satisfies $L_aG_af=f$, and its output has zero mean.
The weighted estimate below controls this specified operator in the
spaces needed by the nonlinear recursion.

\subsection{Combining the zeroth-order and local elliptic bounds}
\label{conv:inverse-subsection}

\begin{lemma}[The inverse bound in the final weighted spaces]
\label{conv:inverse-bound}
For $0<a\le2^{-10}$, the specified inverse extends continuously to
$G_a:Y_a\longrightarrow E_a$, with
\begin{equation}
 \|G_af\|_{E_a}\le M\|f\|_{Y_a},\qquad M=2^{1120}.
 \label{conv:M}
\end{equation}
\end{lemma}

\begin{proof}
{ For $f=0$, linearity gives $G_af=0$. Otherwise
$\|f\|_{Y_a}>0$, legitimizing the strict comparisons below;
the resulting non-strict estimate includes zero.
}
We combine the zeroth-order and local Schauder estimates in
Appendices~\ref{c0:section} and~\ref{sch:sec:schauder} directly.
This avoids an unjustified comparison of global H\"older norms defined using
different sets of pairs. By Proposition~\ref{c0:main}, the zeroth-order
estimate, valid for real or complex smooth mean-zero sources, is
\begin{equation}
 Z_{\mathrm{full}}(G_af):=\sup_Xs_a|G_af|
 \le10^8\sup_Xs_a^3|f|.
 \label{conv:C0-input}
\end{equation}
The local estimates in Proposition~\ref{sch:prop:schauder} use real
six-dimensional charts at scale $s_a(p)$ and
the freezing radius $r_*=2^{-288}$. Any pair in a source chart has physical
distance at most $16r_*s_a(p)$. Weight comparison in these charts gives
$m_{p'q'}\ge s_a(p)/64$. Consequently,
\[
 d_a(p',q')\le1024r_*m_{p'q'}
 =2^{-278}m_{p'q'}
 <\kappa_0m_{p'q'}.
\]
Every such pair belongs to the set in Definition~\ref{conv:def-norms}.
Including the conversion between coordinate and metric distances gives
the explicit local source estimate
\begin{equation}
 \bigl\|s_a(p)^3f\circ\Phi_p\bigr\|_{C^{0,1/2}}
 \le\bigl(64^3+4\cdot64^{7/2}\bigr)\|f\|_{Y_a}
 <2^{24}\|f\|_{Y_a}.
 \label{conv:local-source-conversion}
\end{equation}
Here $\Phi_p$ is the scaled chart, and the equation for
$v_p=s_a(p)u\circ\Phi_p$ is exactly
$\widehat L v_p=s_a(p)^3f\circ\Phi_p$.
The freezing estimate and \eqref{conv:local-source-conversion} give
$H_{0,\mathrm{cap}}\le2^{810}(\|f\|_{Y_a}+Z_{\mathrm{full}}(u))$.
The conversion to the covariant Hessian and the H\"older seminorm defined
by parallel transport costs at most $2^{245}$. The common constant
$2^{1060}$ therefore suffices for
\begin{equation}
 \|u\|_{E_a}\le2^{1060}
       \bigl(\|L_au\|_{Y_a}+Z_{\mathrm{full}}(u)\bigr).
 \label{conv:mixed-Schauder}
\end{equation}
The local estimate was proved using the Hilbert--Schmidt tensor norm;
the operator norm used here is no larger. The estimate also holds for
complex-valued functions.

For the pointwise source term, it is enough to observe that
\[
 \sup_Xs_a^3|f|\le32^3\sup_X\rho_a^3|f|
 \le2^{15}\|f\|_{Y_a}.
\]
Substituting this and \eqref{conv:C0-input} into
\eqref{conv:mixed-Schauder} yields
\[
 \|G_af\|_{E_a}
 \le2^{1060}(1+2^{15}\cdot10^8)\|f\|_{Y_a}
 <2^{1120}\|f\|_{Y_a}.
\]
The set of pairs for the uncapped output contains the set for the capped
output, and its weights are larger. The argument does not assert a
termwise comparison of the two global source H\"older norms.
The estimate on smooth sources extends continuously to the closure $Y_a$,
with values in $E_a$.
{ If $f_n\to f$ in $Y_a$, the estimate makes $G_af_n$ Cauchy
and its limit independent of the approximants.
Trace contraction and $\nabla J=0$ give
$\|L_au\|_{Y_a}\le3\|u\|_{E_a}$.
Thus $L_a$ extends to the closures and
$L_aG_af=\lim_nL_aG_af_n=f$.
Mean zero and commutation with conjugation also pass to the limit.
}
\end{proof}

The inverse now acts on one fixed pair of capped Banach spaces. We use
these same spaces for the source and both nonlinear polarizations.

\section{The cubic equation and the Ricci-flat metric}
\label{conv:section}
The background and its volume source are fixed. This section estimates the
exact determinant polarizations in the capped norms, verifies explicit
smallness inequalities, and sums the source-degree recursion.

\subsection{The exact cubic equation}
\label{conv:nonlinear-bounds}

The relative complex Hessian is the endomorphism $A_a(u)$ defined in
\eqref{conv:linearization}. Define the symmetric bilinear and trilinear operators
\begin{equation}
 Q_a(u,v)=\frac{\omega_a^{\mathrm{bg}}\wedge(\ii\partial\bar\partial u)
                         \wedge(\ii\partial\bar\partial v)}{2\dd V_a},
 \qquad
 T_a(u,v,w)=\frac{(\ii\partial\bar\partial u)
                  \wedge(\ii\partial\bar\partial v)
                  \wedge(\ii\partial\bar\partial w)}{6\dd V_a}.
 \label{conv:QT}
\end{equation}
The determinant identity in complex dimension three is exactly
\begin{equation}
 \det\bigl(\Id+A_a(u)\bigr)
 =1+L_au+Q_a(u,u)+T_a(u,u,u).
 \label{conv:determinant}
\end{equation}
The prescribed-volume equation is therefore equivalent to the cubic
equation
\begin{equation}
 L_au+Q_a(u,u)+T_a(u,u,u)=\mathcal R_a.
 \label{conv:cubic}
\end{equation}
\subsection{Mean-zero products and explicit norm constants}

\begin{lemma}[Mean-zero identities and explicit product constants]
\label{conv:products}
On the complexified spaces of Definition~\ref{conv:def-norms},
\begin{align}
 \|A_a(u)\|_\infty&\le a^{-3}\|u\|_{E_a},
       \label{conv:Hessian-bound}\\
 \|Q_a(u,v)\|_{Y_a}&\le3a^{-3}\|u\|_{E_a}\|v\|_{E_a},
       \label{conv:Q-bound}\\
 \|T_a(u,v,w)\|_{Y_a}&\le a^{-6}
             \|u\|_{E_a}\|v\|_{E_a}\|w\|_{E_a}.
       \label{conv:T-bound}
\end{align}
Every mixed term involving $Q_a,T_a$ has zero $\dd V_a$ integral.
These operators act continuously on, and preserve, the chosen closures.
\end{lemma}

\begin{proof}
Raise one index of the real Hessian using the metric to obtain an
endomorphism $B$. Its $J$-invariant projection is
\[
 \Pi_JB=\frac12(B+J^{-1}BJ).
\]
Since $J$ is an isometry, this projection is still a contraction in the
complexified operator norm. In a K\"ahler unitary frame, $A_a(u)$ is the
corresponding complex three-dimensional block of $\Pi_JB$, with the
normalization $L_a=\Delta_{g_a}/2$. Thus
$\|A_a(u)\|_{\mathrm{op}}\le\|\nabla^2u\|_{\mathrm{op}}$.
Parallel transport preserves both the metric and $J$, so the same
contraction bound applies to the difference between tensors at two points.
Using $\rho_a\ge a$ gives \eqref{conv:Hessian-bound}.

For three complex $3\times3$ matrices, let $\mathfrak D(A,B,C)$ be the
symmetric trilinear polarization of the determinant: the average of the
$6$ determinants obtained by selecting one column from each matrix.
The matrices are assigned in all orders to the three column positions.
Hadamard's column inequality gives
\[
 |\mathfrak D(A,B,C)|\le
 \|A\|_{\mathrm{op}}\|B\|_{\mathrm{op}}\|C\|_{\mathrm{op}}.
\]
The pointwise expressions for \eqref{conv:QT} are
$Q(A,B)=3\mathfrak D(A,B,\Id)$ and
$T(A,B,C)=\mathfrak D(A,B,C)$. Their pointwise constants are therefore
$3$ and $1$, respectively, without any commutation or reality assumption
on the matrices.

Write $H_0(u)=\sup\rho_a^3\|\nabla^2u\|_{\mathrm{op}}$ and
$H_\alpha(u)=[\nabla^2u]_*$.
Using the minimum weight $m_{pq}$ at each pair and expanding the bilinear
difference into two products gives
\[
 \|Q_a(u,v)\|_{Y_a}\le3a^{-3}
 \bigl(H_0(u)H_0(v)+H_\alpha(u)H_0(v)+H_0(u)H_\alpha(v)\bigr)
 \le3a^{-3}\|u\|_{E_a}\|v\|_{E_a}.
\]
The trilinear difference has three terms. Together with the supremum term,
they are bounded by $a^{-6}\prod(H_0+H_\alpha)$, proving
\eqref{conv:T-bound}. These estimates already include the tangent
elliptic direction: the products are products of complex three-dimensional
matrices, so no further factor for that direction is needed.
{At minimum weight $m$, unchanged Hessians cost $m^{-3}H_0$
and a changed factor costs $m^{-7/2}d^{1/2}H_\alpha$.
Multiplication by $m^{7/2}d^{-1/2}$ leaves $m^{-3}\le a^{-3}$
or $m^{-6}\le a^{-6}$. The supremum terms have the same bounds;
telescoping and expanding $\prod(H_0+H_\alpha)$ give the stated constants.
}

Each $\ii\partial\bar\partial u$ is a globally exact closed two-form,
and $\omega_a^{\mathrm{bg}}$ is closed. Stokes' theorem gives zero integral for every
mixed evaluation of $Q_a,T_a$. This identity is first established for smooth
functions and then extended using the product bounds and the closure
definition. Smooth inputs give smooth outputs; continuity consequently
also proves preservation of the closures.
\end{proof}

\subsection{The source bound}

\begin{lemma}[The source constant]
\label{conv:source-bound}
For $0<a\le1/128$,
\begin{equation}
 \|\mathcal R_a\|_{Y_a}<864a^4.
 \label{conv:CR}
\end{equation}
\end{lemma}

\begin{proof}
For completeness, the sharp first derivative of the cutoff satisfies
$0\le H'\le2$. To see this, set $z=|2t-1|$ and
$y=2z/(1-z^2)$ for $0<t<1$. Direct differentiation gives
\[
 H'(t)=\frac{2(1+z^2)}{(1-z^2)^2\cosh^2y}
 \le\frac2{1+z^2}\le2,
\]
because $\cosh^2y\ge1+y^2=(1+z^2)^2/(1-z^2)^2$.
The two transition supports of $\psi$ are disjoint, and each transition
has length $1/128$. Hence $|\psi'|\le256$. On this support,
$r\ge r_1\ge a$. The radial coefficients of the background satisfy
\[
 p\le\sqrt2,\qquad q\ge\frac{64}{65\sqrt2},
 \qquad |\nabla r|=q^{-1/2}\le\sqrt2.
\]
Set $B_0=128^4/65$, so that $\beta_a\le B_0a^4$.
Using $h_a\ge64/65$, $0<c_a\le1$, and \eqref{geom:source}, we obtain
\begin{align*}
 \sup|\mathcal R_a|&\le
       \frac{65}{64}(B_0+16\pi^2)a^4=:C_0a^4,\\
 \sup|\nabla\mathcal R_a|&\le
       B_0\left(\frac{65}{64}\right)^2 256\sqrt2\,a^4=:C_1a^4.
\end{align*}
The global volume-normalization term $16\pi^2a^4$ is included here.
Disjointness of the tubes eliminates any additional factor of $16$.
For the chosen pairs, $d_a(p,q)\le m_{pq}/4$, so
\[
 [\mathcal R_a]_*\le\frac12C_1c^4a^4,
 \qquad \sup\rho_a^3|\mathcal R_a|\le C_0c^3a^4.
\]
It follows that
\[
 \|\mathcal R_a\|_{Y_a}
 \le\left(128+520\sqrt2+\frac{65\pi^2}{131072}\right)a^4<864a^4.
\]
The final comparison is certified by the rational upper bounds
$\sqrt2<99/70$ and $\pi<22/7$.
\end{proof}

\subsection{Convergence on an explicit parameter interval}
\label{conv:theorem-subsection}

\begin{theorem}[The exact Ricci-flat metric on the special Schoen model]
\label{conv:main-theorem}
Set $M=2^{1120}$. For every
\begin{equation}
                     0<a\le2^{-2260},
 \label{conv:a-interval}
\end{equation}
define
\begin{align}
 U_{a,1}&=G_a\mathcal R_a,
        \label{conv:recursion-first}\\
 U_{a,n}&=-G_a\left[
       \sum_{\substack{i+j=n\\i,j\ge1}}Q_a(U_{a,i},U_{a,j})
       +\sum_{\substack{i+j+k=n\\i,j,k\ge1}}
            T_a(U_{a,i},U_{a,j},U_{a,k})\right],\qquad n\ge2.
        \label{conv:recursion}
\end{align}
The sums are over ordered positive indices. For $N=0$, a truncated
series below is interpreted as zero. The series
$u_a=\sum_{n\ge1}U_{a,n}$ converges in $E_a$, and
\begin{equation}
 \omega_a^{\mathrm{CY}}=\omega_a^{\mathrm{bg}}+\ii\partial\bar\partial u_a
 \quad\text{satisfies}\quad
 \frac{(\omega_a^{\mathrm{CY}})^3}{6}=c_a\mu,
 \qquad \int_Xu_a\,\dd V_a=0.
 \label{conv:exact-CY}
\end{equation}
This is a globally smooth positive definite Ricci-flat metric in the
K\"ahler class \eqref{geom:kahler-class}, and it is unique in that class.

For $n\ge1$ and every integer $N\ge0$,
\begin{align}
 \|U_{a,n}\|_{E_a}&\le1728Ma^4(4/5)^n,
        \label{conv:coefficient-bound}\\
 \left\|u_a-\sum_{n=1}^NU_{a,n}\right\|_{E_a}
       &\le8640Ma^4(4/5)^{N+1},
        \label{conv:potential-tail}\\
 \left\|(\omega_a^{\mathrm{bg}})^{-1}\ii\partial\bar\partial
          \left(u_a-\sum_{n=1}^NU_{a,n}\right)\right\|_\infty
       &\le8640Ma(4/5)^{N+1}.
        \label{conv:metric-tail}
\end{align}
The full relative correction also satisfies
\begin{equation}
 \|A_a(u_a)\|_\infty\le1728Ma<\frac1{24},
 \qquad
 \frac{23}{24}\omega_a^{\mathrm{bg}}<\omega_a^{\mathrm{CY}}
                      <\frac{25}{24}\omega_a^{\mathrm{bg}}.
 \label{conv:positive-comparison}
\end{equation}
\end{theorem}

\begin{proof}
Let
\[
 \epsilon_a=\|G_a\mathcal R_a\|_{E_a}.
\]
Lemmas~\ref{conv:inverse-bound} and~\ref{conv:source-bound} give
$\epsilon_a\le864Ma^4$.
By Lemma~\ref{conv:products}, we may choose
\begin{equation}
 K_2=3Ma^{-3},\qquad K_3=Ma^{-6},\qquad C_H=a^{-3},
 \label{conv:constants}
\end{equation}
so that $\|G_aQ_a(u,v)\|_{E_a}\le K_2\|u\|_{E_a}\|v\|_{E_a}$ and
$\|G_aT_a(u,v,w)\|_{E_a}\le K_3\|u\|_{E_a}\|v\|_{E_a}\|w\|_{E_a}$.
At the upper endpoint of \eqref{conv:a-interval}, exact arithmetic gives
\begin{align}
 4K_2\epsilon_a+12K_3\epsilon_a^2
 &\le10368M^2a+8957952M^3a^2\notag\\
 &\le\frac{81}{8192}+8957952\cdot2^{-1160}
  <\frac{41}{4096}<\frac12,
       \label{conv:contraction-arithmetic}\\
 2C_H\epsilon_a&\le1728Ma
  \le1728\cdot2^{-1140}<\frac1{24}.
       \label{conv:positivity-arithmetic}
\end{align}
{The displayed upper bounds $10368M^2a+8957952M^3a^2$ and $1728Ma$ increase with $a>0$, so their values on the stated interval are largest at its upper endpoint.} The interval also lies within the background
positivity range $a\le1/128$ and the geometric range $a\le2^{-10}$.

For each fixed $a>0$, introduce a complex source-degree parameter $t$ and set
\[
 \mathcal F_t(u)=tG_a\mathcal R_a-G_aQ_a(u,u)-G_aT_a(u,u,u).
\]
On $\|u\|_{E_a}\le2\epsilon_a$ and $|t|\le3/2$, the preceding constants give
\begin{align*}
 \|\mathcal F_t(u)\|_{E_a}
 &\le\tfrac32\epsilon_a+4K_2\epsilon_a^2+8K_3\epsilon_a^3
 \le2\epsilon_a,\\
 \|\mathcal F_t(u)-\mathcal F_t(v)\|_{E_a}
 &\le(4K_2\epsilon_a+12K_3\epsilon_a^2)\|u-v\|_{E_a}
 \le\tfrac12\|u-v\|_{E_a}.
\end{align*}
{ For the iterates $u_j$ from zero,
$\sup_{|t|\le3/2}\|u_{j+1}-u_j\|_{E_a}
\le2^{-j}(3\epsilon_a/2)$.
Uniform convergence passes their Cauchy formula to a Bochner integral,
which is norm differentiable inside each circle. At $t=0$,
uniqueness identifies the fixed point as zero.
}
Iteration from zero stays in the closed radius-$2\epsilon_a$ ball and
converges uniformly for $|t|\le3/2$. Each iterate is a Banach-valued
polynomial in $t$. Its limit $u(t)$ is therefore holomorphic for
$|t|<3/2$, satisfies
$\|u(t)\|_{E_a}\le2\epsilon_a$, and has $u(0)=0$.
Continuity of the bilinear and trilinear operators permits substitution
of its norm-convergent Taylor series. Comparing powers of $t$ yields
\eqref{conv:recursion-first}--\eqref{conv:recursion}; each nonlinear term
at a given order uses only lower-order coefficients. The Stokes identities
ensure the required mean-zero condition at every application of $G_a$.
{Indeed, Cauchy's formula gives
$\|U_{a,n}\|_{E_a}\le2\epsilon_a r^{-n}$ for $1<r<3/2$.
The double and triple Cauchy products are therefore absolutely
convergent at $t=1$; their degree-$n$ factors have positive indices
strictly below $n$.
}

Banach-valued Cauchy estimates on $|t|=5/4$ give
\[
 \|U_{a,n}\|_{E_a}\le2\epsilon_a(4/5)^n,
 \qquad
 \left\|u(1)-\sum_{n=1}^NU_{a,n}\right\|_{E_a}
 \le10\epsilon_a(4/5)^{N+1}.
\]
Combining these estimates with $\epsilon_a\le864Ma^4$ and
\eqref{conv:Hessian-bound} proves the coefficient and remainder bounds.

At $t=1$, complex conjugation preserves the data, the operators, and the
fixed-point equation. Uniqueness of the fixed point makes
$u_a=u(1)$ real-valued. Estimate~\eqref{conv:positivity-arithmetic}
ensures strict positivity and proves \eqref{conv:positive-comparison}.
The fixed-point equation and \eqref{conv:determinant} then give exactly
\[
 \det(\Id+A_a(u_a))=1+\mathcal R_a=\frac{c_a}{h_a},
\]
which is equivalent to \eqref{conv:exact-CY}.
{The following identity is initially distributional: its volume ratio
is the positive constant $c_a$. The bootstrap below makes it classical.
}
Since $\Omega$ is holomorphic and nowhere vanishing,
\[
 \Ric(\omega_a^{\mathrm{CY}})
 =-\ii\partial\bar\partial
       \log\frac{(\omega_a^{\mathrm{CY}})^3/6}{\mu}
 =-\ii\partial\bar\partial\log c_a=0.
\]
{Fix $a>0$ and a relatively compact chart, and write the equation as
$\mathcal F(x,D^2u_a)=0$, using the logarithmic determinant minus
the prescribed logarithmic density. The bound
$\|A_a(u_a)\|_\infty<1/24$ makes $\mathcal F_M$ uniformly elliptic
near the compact Hessian graph: its rank-one contraction is bounded
by positive multiples of the squared vector length.
Constants here may depend on $a$.
}
The solution initially belongs to $C^{2,1/2}$. Positivity makes the
Monge--Amp\`ere equation uniformly elliptic for each fixed $a>0$;
local difference-quotient estimates and Schauder bootstrapping with the
smooth background and source give a $C^\infty$ solution.
For $w_h=(u_a(x+he)-u_a(x))/h$, subtraction gives
$A_h:D^2w_h=b_h$, where
\[
 A_h=\int_0^1\mathcal F_M
 \bigl(x,D^2u_a(x)+\tau[D^2u_a(x+he)-D^2u_a(x)]\bigr)\,\dd\tau
\]
and $b_h=-\int_0^1\mathcal F_{x_e}
(x+\tau he,D^2u_a(x+he))\,\dd\tau$.
For small $h$, ellipticity and the $C^{0,\alpha}$ bounds of
$A_h,b_h$ are uniform in $h$, since $u_a\in C^{2,\alpha}$,
$\alpha=1/2$. Also $\|w_h\|_\infty\le\|Du_a\|_\infty$.
The linear Schauder estimate gives uniform interior $C^{2,\alpha}$ bounds.
Compactness and $w_h\to\partial_eu_a$ in $C^1$ identify the $C^2$
limit; its H\"older bound persists, giving $u_a\in C^{3,\alpha}$.

Differentiating once gives a linear equation for $\partial_eu_a$
with $C^{1,\alpha}$ coefficients and source.
Repeating the difference-quotient estimate on smaller charts gives
all higher derivatives. These are estimates for the new linearized
coefficients, with constants allowed to depend on $a$.

A global potential correction preserves the K\"ahler class, and
Calabi--Yau uniqueness \cite{Yau} gives the unique metric in that class.
The zero-mean condition fixes the additive constant of the potential.
{Indeed, equal forms give $\Delta_{g_a}(u-v)=0$.
The Dirichlet energy makes $u-v$ constant on connected $X$;
equal zero means make that constant zero.
}
\end{proof}

\section{Truncation and the scope of the representation}
\label{computational:section}
The theorem supplies an exact infinite representation and an explicit
source-degree truncation error. Its coefficient functions are evaluated
through the specified local and global linear operations. A finite
approximation has several separate error terms.

The local Frobenius initialization has the value and derivative tails in
\eqref{kernel:frobenius-tail}. On each subsequent radial interval, the
ordered-integral expansion \eqref{kernel:volterra} has the stated
$e^BB^{N+1}/(N+1)!$ transfer tail. The boundary inverse has the energy
tail \eqref{kernel:boundary-neumann}, and its reconstructed trace has
\eqref{kernel:trace-layer-tail}. Finally,
\eqref{conv:potential-tail} and \eqref{conv:metric-tail} control the
nonlinear potential and relative metric after source-degree truncation.
Certified angular and tangential truncation additionally requires bounds
for the tails of the chosen source representation and their propagation
through the inverse. The frequency decay estimates in
Appendix~\ref{c0:section} are an input to that analysis.
Numerical evaluation also requires error bounds
for the remaining heat-lattice sums, radial integrals, and boundary
quadrature. These errors must be included when a finite metric value is
claimed.
These additional error bounds have not been assembled into a complete
finite numerical evaluation in this paper.

The expansion parameter is source degree at a fixed $a>0$.
The quantity $\epsilon_a=\|G_a\mathcal R_a\|_{E_a}$ measures the source
after application of the inverse; it is distinct from the resolution
scale. The coefficients depend on $a$, and no Taylor analyticity at the
degenerating endpoint $a=0$ is asserted.
The sufficient interval \eqref{conv:a-interval} concerns this particular
background, norm, and family of classes. Its endpoint is below the
positive range of ordinary IEEE double-precision numbers; evaluating at
that scale would require arbitrary-precision or appropriately scaled
arithmetic. The proof does not provide a finite-precision metric grid or
an optimal convergence range.

The numerical certificates below consist of displayed inequalities and
rational majorants for finite constant comparisons. They accompany the
analytic arguments for the kernels, transmission, local geometry, and
elliptic estimates. The quantitative conclusion is the one established by
those proofs; no external checking program is required as an input.

\appendix
\section{The Schoen weighted zeroth-order inverse certificate}\label{c0:section}

This appendix gives an explicit zeroth-order constant for the global linear
inverse. We use the full geometric weight $s_a$ from the main text, rather
than the capped weight $\rho_a=\min\{s_a,c\}$ used in the nonlinear estimates.
The result does not include a full Schauder or Hessian estimate.

Use the fixed radii \eqref{geom:radii} and assume $0<a\le1/128$.
Write $x=D^{1/4}$ for the moment radius inside each tube, to distinguish it
from the global weight. On the sixteen resolved tubes $M_j$,
$s_a=x\in[a,b]$; on the flat exterior, $b\le s_a\le1$.
The full weight is the one specified in Section~\ref{weighted:section}
and Appendix~\ref{sch:sec:schauder}. Set
\begin{equation}\label{c0:source-norm}
 A_f=\sup_X s_a^3|f|,\qquad
 \|u\|_{C^0_{-1},s_a}=\sup_X s_a|u|.
\end{equation}
{Only the weight formula from Appendix~\ref{sch:sec:schauder} is used
here: $D^{1/4}$ increases from $a$ to $b$, and $(1-H)r+H\in[r,1]$.
This remains valid for $a\le1/128$, independently of the chart
estimates restricted to $a\le2^{-10}$.
}
\begin{proposition}[An explicit weighted zeroth-order constant]\label{c0:main}
For the product background on the special Schoen threefold constructed in
the main text, let $f$ be a smooth real- or complex-valued source satisfying
$\int_X f\,\dd V_a=0$. The global zero-mean inverse $G_a$ satisfies
\begin{equation}\label{c0:main-bound}
 \sup_X s_a|G_af|\le10^8\sup_Xs_a^3|f|,
 \qquad 0<a\le\frac1{128}.
\end{equation}
The constant $10^8$ is independent of $a$.
\end{proposition}

\begin{proof}\leavevmode\par
{For $f=0$ the linear construction vanishes.
Otherwise $A_f>0$, so the strict numerical comparisons multiplied
by $A_f$ or $A_f^2$ below are valid.
}
\paragraph{A local barrier and the source flux.}
Let $E=\mathbb C/(\mathbb Z+\ii\mathbb Z)$, with area $1$.
Inside a tube, the background metric and operator are
\[
 g_a=F_a^{-1}\dd x^2+x^2g_H+x^2F_a\eta^2+|\dd w|^2,
 \qquad
 \dd V_a=x^3\dd x\,\dd V_{\mathbb{RP}^3}\,\dd V_E,
 \qquad L_a=\frac12\Delta_{g_a}.
\]
Regular radial modes satisfy $x^3F_av'\to0$ at the bolt. The volume-controlled
background has $\beta_a\le1/65$ and
\begin{equation}\label{c0:F-lower}
 F_a(x)\ge(1-\beta_a)(1-a^4/x^4).
\end{equation}
Indeed, $D=r^4+a^4-4\beta_a J(r)$, $0\le4\beta_a J(r)\le a^4$, and
$F_a=r^4(1-\beta_a\psi)/D$. Since $D-a^4\le r^4$, the inequality follows.
On the collar $c\le x\le b$, one has $F_a=1$ and $x=r$.

Define
\begin{equation}\label{c0:barrier}
 \mathcal B_a(x)=2\int_x^b\frac{t-a}{t^3F_a(t)}\,\dd t.
\end{equation}
The radial operator is $(2x^3)^{-1}\partial_x(x^3F_a\partial_x)$, so
$-L_a\mathcal B_a=x^{-3}$ and $\mathcal B_a(b)=0$.
The numerator cancels the zero of $F_a$ at the bolt; consequently
$\mathcal B_a$ is regular and smooth there. By \eqref{c0:F-lower},
\[
 \frac{2(t-a)}{t^3F_a(t)}
 \le\frac{2}{1-\beta_a}\frac{t}{(t+a)(t^2+a^2)}
 \le\frac{2}{1-\beta_a}t^{-2}.
\]
It follows that $x\mathcal B_a(x)\le65/32$.
For the local solution $L_aw_j=f$ with zero boundary values, the maximum
principle and the boundary derivative limit give
\begin{equation}\label{c0:local-source}
 |w_j|\le A_f\mathcal B_a,
 \qquad
 |j_j|:=|\partial_nw_j|
 \le\frac{2(b-a)}{b^3}A_f\le512A_f.
\end{equation}
For complex sources, apply the same comparison to
$\Re(e^{\ii\vartheta}w_j)$ for every constant phase $\vartheta$; the constants
remain unchanged. Since $\Vol(\mathbb{RP}^3)=\pi^2$, we also have
\begin{equation}\label{c0:weighted-local-volume}
 \sum_{j=1}^{16}\int_{M_j}x^{-1}\,\dd V_a
 =\frac{16\pi^2}{3}(b^3-a^3)
 \le\frac{\pi^2}{768}
 <\frac{121}{9408}<\frac{13}{1000}.
\end{equation}
\paragraph{The collar estimates used below.}
The full relative DtN estimate, the flat radial equation, the Riccati
bounds, and the Dirichlet entrance-limit formula were proved in
Section~\ref{kernel:global-section}; see
\eqref{c0:relative-DN}, \eqref{c0:alpha-bounds},
\eqref{c0:dirichlet-extreme}, and \eqref{c0:angular-q}.
We now quantify the additional decay needed to sum all tangential
frequencies.

\paragraph{Mass-frequency exponential decay and the full Fourier sum.}
Summing only the frequency-independent bound \eqref{c0:relative-DN} over
$n$ would diverge. For the additional decay, compare $\alpha$ with the line
$z-4$ for $z\ge4$. At $z=4$, $\alpha\ge0$; at a contact point,
\eqref{c0:riccati} gives
$\alpha_z=[\ell(\ell+2)+6z-8]/z\ge4>1$. Thus $\alpha\ge z-4$.
If $\kappa\ge256$, then $\kappa c\ge8$, and
$v_0'/v_0\ge\kappa/2$ throughout the collar.
Set $d=(b+c)/2=3b/4$. Then
\[
 \int_c^b\frac{\dd t}{t^3v_0(t)^2}
 \ge\frac{d-c}{b^3v_0(d)^2},\qquad
 \frac{v_0(d)}{v_0(b)}\le e^{-\kappa(b-d)/2}.
\]
Substituting into \eqref{c0:dirichlet-extreme} and using
$\lambda_0\ge\kappa/2$, we get
\begin{equation}\label{c0:mass-q}
 \frac{|\delta|}{\lambda_0}
 \le q_{\rm mass}(\kappa)
 :=\frac8{\kappa b}e^{-\kappa b/4}
 =\frac{128}{\kappa}e^{-\kappa/64},\qquad\kappa\ge256.
\end{equation}
No unbounded or unspecified Bessel asymptotic remainder is used.
{For $(\ell,\kappa)\ne(0,0)$, define the finite Hopf envelope
\[
 q_{\ell,\kappa}
 :=\max_{m=-\ell,-\ell+2,\ldots,\ell}
 \frac{|\delta_{\ell,m,\kappa}|}{\lambda_{0,\ell,\kappa}},
 \qquad \lambda_{0,\ell,\kappa}>0.
\]
The preceding bounds are uniform in $m$:
$q_{\ell,\kappa}\le q_\ell$ for $\ell\ge2$,
$q_{0,\kappa}\le1/15$ for $\kappa>0$, and
$q_{\ell,\kappa}\le q_{\rm mass}(\kappa)$ for $\kappa\ge256$.
Thus $q_{\ell,\kappa}^2\le q_\ell q_{\rm mass}(\kappa)$
in the positive-angular high-mass range.
}

On $\mathbb{RP}^3_b$, the sum of squared values of the orthonormal harmonics
of degree $\ell$ is $(\ell+1)^2/(\pi^2b^3)$. The normalized Fourier functions
on $E$ have modulus $1$. The addition formula and weighted
Cauchy--Schwarz therefore give
\begin{equation}\label{c0:smoothing-CD}
 \|D_Nh\|_\infty\le C_D\langle h,N_0h\rangle^{1/2},\qquad
 C_D^2\le\frac1{\pi^2b^3}
 \sum_{\substack{n\in\mathbb Z^2,\ \ell\in2\mathbb Z_{\ge0}\\(\ell,n)\ne(0,0)}}(\ell+1)^2\lambda_0q_{\ell,\kappa}^2.
\end{equation}
The zero mode is omitted. Each boundary point belongs to a single tube,
whose energy is bounded by the total energy; no additional factor $16$
occurs here.

For low masses $\kappa<256$, the inequality $\pi>3$ leaves at most
$85^2=7225$ indices $n$. Enlarge the sum over positive even $\ell$ to all
integers $\ell\ge2$. By \eqref{c0:alpha-bounds} and
\eqref{c0:angular-q}, the resulting majorant has first term $2592/441$ and
successive-term ratio at most $19/162<1/8$. Consequently,
\[
 \sum_{\ell\ge2}(\ell+1)^2(16\ell+256)q_\ell^2
 \le\frac{2592}{441}\frac87<7.
\]
The term with $\ell=0,n\ne0$ is at most $256/225<2$. The entire low-mass
contribution is therefore less than $65025$.

For high masses $\kappa\ge256$, the two relative bounds imply
$q_{\ell,\kappa}^2\le q_\ell q_{\rm mass}$. The contribution from
$\ell\ge2$ is thus at most
\[
 e^{-\kappa/64}
 \sum_{\ell\ge2}(\ell+1)^2(8\ell+128)q_\ell
 <124e^{-\kappa/64}.
\]
The last series has first term $1296/21$ and successive-term ratio at most
$38/81<1/2$. The term with $\ell=0$ is less than $9e^{-\kappa/64}$.
Since $\kappa\ge3(|n_1|+|n_2|)$ and $e^t-1\ge t$,
\[
 \sum_{n\in\mathbb Z^2}e^{-\kappa/64}
 \le\left(1+\frac2{e^{3/64}-1}\right)^2
 \le\left(\frac{131}{3}\right)^2=\frac{17161}{9}.
\]
The full sum has the explicit rational upper bound
\begin{equation}\label{c0:CD-numerical}
 C_D^2\le\frac{4096}{9}
 \left(65025+\frac{133\cdot17161}{9}\right)
 =\frac{11745845248}{81}<13000^2.
\end{equation}
\paragraph{Absolute integral bounds for the flat kernel.}
Let $A=E^3$, and let $Y=A/(\mathbb Z_2\times\mathbb Z_2)$ be the flat
quotient from the main text. Its volume is $1/4$. Denote by $G_A,G_Y$ the
zero-mean Green kernels for $L_0=\Delta/2$, and set
\[
 K_G=\sup_x\int_Y|G_Y(x,y)|\,\dd V_0(y),\qquad
 K_\Sigma=\sup_x\int_\Sigma|G_Y(x,y)|\,\dd S_y.
\]
Split the heat integral at $T=1/4$:
\[
 G_A(x,y)=-\sum_{q\in\mathbb Z^6}P_T(x-y-q)+T-\mathcal E_T(x-y),
 \qquad
 P_T(v)=\frac{(1+|v|^2/(2T))e^{-|v|^2/(2T)}}{2\pi^3|v|^4}.
\]
The long-time tail obeys
\[
 \|\mathcal E_T\|_\infty\le S_T
 :=\frac1{2\pi^2}
 \left[\left(1+\frac{2e^{-2\pi^2T}}{1-e^{-6\pi^2T}}\right)^6-1\right].
\]
The inequalities $\pi>3$ and the exact Taylor lower bound
\[
 \sum_{j=0}^8\frac{(9/2)^j}{j!}
 =\frac{99082729}{1146880}>80
\]
yield
\begin{equation}\label{c0:heat-tail}
 S_T\le\frac{(40/39)^6-1}{18}
 =\frac{577256239}{63337387698}<\frac1{100}.
\end{equation}
The positive short-time kernel has integral $T$ on the entire torus. Hence
$\sup_x\int_A|G_A|\le2T+S_T<51/100$. The four-term quotient kernel and
the degree-four volume covering cancel after integration, so
\begin{equation}\label{c0:KG}
 K_G<\frac{51}{100}<1.
\end{equation}
The sixteen quotient boundaries lift to thirty-two copies of
$E\times S_b^3$ in $A$, with sixteen in each of two families. The triangle
inequality and the degree-four boundary covering allow the integral of
$|G_A|$ over these lifted boundaries to serve as an upper bound for
$K_\Sigma$. Integrating the periodic heat kernel along the tangential $E$
gives exactly $1$. The short-time six-dimensional kernel therefore reduces
to the four-dimensional normal kernel
\[
 P_T^{(4)}(z)=\frac{e^{-|z|^2/(2T)}}{2\pi^2|z|^2}.
\]
For each family, the periodic normal centers form a translate of
$(\tfrac12\mathbb Z)^4$, with minimum spacing $1/2=2h$, where $h=1/4$.
The four-dimensional Newton shell identity is
\[
 \int_{S_b^3(p)}\frac{\dd S_y}{2\pi^2|x-y|^2}
 =\frac{b^3}{\max\{b,|x-p|\}^2}.
\]
There is at most one center at distance $r<h$, contributing at most $b$.
A center at distance $r\ge h$ contributes at most
$F_T(r)=b^3e^{-(r-b)^2/(2T)}/r^2$. The radius-$h$ balls around centers are
disjoint, so the counting function satisfies $N_x(r)\le(1+r/h)^4$.
Since $F_T$ is decreasing, Tonelli's theorem and integration by parts give
\begin{align*}
 \sum_{|x-p|\ge h}F_T(|x-p|)
 &\le\int_h^\infty(1+r/h)^4(-F_T'(r))\,\dd r\\
 &\le\frac{16b^3}{h^2}
  +\frac{32b^3}{h^4}\left(T+b\sqrt{\frac{\pi T}{2}}\right).
\end{align*}
The total lifted boundary area is $64\pi^2b^3$. Thus
\begin{align}
 K_\Sigma
 &\le2\left[b+\frac{16b^3}{h^2}
 +\frac{32b^3}{h^4}\left(T+b\sqrt{\frac{\pi T}{2}}\right)\right]
 +64\pi^2b^3(T+S_T)\nonumber\\
 &<\frac{17}{12}+\frac{121}{784}\frac{13}{50}
 =\frac{171319}{117600}<\frac32<2.\label{c0:KS}
\end{align}
Here we used $\pi<22/7$ and $\sqrt{\pi/8}<2/3$. The tangential directions
of the two families differ, but each can be integrated separately; no
additional counting estimate between the two families is needed.

\paragraph{The global source energy and the trace energy.}
Let $S_{\rm layer}\tau=-\tfrac12\int_\Sigma G_Y\tau\,\dd S$,
let $S$ be its boundary trace, and set $A_0=N_0+N_{\rm out}$.
The exterior DtN map $N_{\rm out}$ is used only for the energy argument,
not as a computational input. The identity $A_0S=\Id$ applies only to
densities of total integral zero, with trace functions taken modulo a
common constant. Define
\[
 v_{\rm ext}(x)=\int_{\mathcal E}G_Y(x,y)f(y)\,\dd V_0(y),
 \qquad W=v_{\rm ext}-S_{\rm layer}j,
 \qquad h_0=W|_\Sigma.
\]
Since $s_a\ge b$ outside the tubes, $|f|\le4096A_f$ there. From
\eqref{c0:local-source}, \eqref{c0:KG}, and \eqref{c0:KS},
\begin{equation}\label{c0:KW}
 \|W\|_\infty\le K_WA_f,
 \qquad K_W=b^{-3}K_G+256K_\Sigma<4608.
\end{equation}
The local divergence identity is
$\int_{M_j}f\,\dd V_a=\tfrac12\int_{\Sigma_j}j_j\,\dd S$.
Together with the zero-mean condition on the source, this cancels the
flat background constants exactly. Thus $W$ is harmonic inside the flat
tubes and satisfies $L_0W=f$ in the exterior; its distributional source
also contains $j\delta_\Sigma/2$. The real or complex energy pairing gives
\begin{align*}
 \int_Y|\nabla W|^2\,\dd V_0
 &\le\|W\|_\infty
       \left(2\int_{\mathcal E}|f|\,\dd V_0+\int_\Sigma|j|\,\dd S\right)\\
 &\le K_WV_{\rm src}A_f^2,
 \qquad
 V_{\rm src}\le2048+2\pi^2<2070.
\end{align*}
We used $\Vol_0(\mathcal E)\le1/4$ and
$\operatorname{Area}(\Sigma)=16\pi^2b^3$.

The required trace norm is $\|h_0\|_{E_0}^2=\langle h_0,A_0h_0\rangle$,
not merely the interior $N_0$ energy. The interior harmonic extension is
$W$. The exterior harmonic extension minimizes the Dirichlet energy among
functions with the same trace. Consequently,
\begin{equation}\label{c0:trace-energy}
 \|h_0\|_{E_0}^2\le\int_Y|\nabla W|^2\,\dd V_0
 <4608\cdot2070\,A_f^2,
 \qquad \|h_0\|_{E_0}<3200A_f.
\end{equation}
No unspecified exterior Poincar\'e or Schauder constant enters this bound.

\paragraph{The boundary inverse, mean normalization, and the final bound.}
By \eqref{c0:relative-DN} and nonnegativity of $N_{\rm out}$,
$\|A_0^{-1}D_N\|_{E_0}\le1/15$. In the quotient energy space,
the trace class satisfies $[h]=(\Id+A_0^{-1}D_N)^{-1}[h_0]$.
The actual trace is fixed by the computable density representation
\[
 \tau=\sum_{q\ge0}(-D_NS)^qD_Nh_0,
 \qquad h=h_0-S\tau,\qquad\tau=D_Nh.
\]
Since $D_N$ annihilates the sixteen local constants and is self-adjoint,
$\tau$ has integral zero on every boundary component. Equations
\eqref{c0:smoothing-CD}, \eqref{c0:CD-numerical}, and
\eqref{c0:trace-energy} imply
\[
 \|\tau\|_\infty
 \le13000\|h\|_{E_0}
 \le13000\frac{15}{14}(3200)A_f
 =\frac{312000000}{7}A_f<44571429A_f.
\]
The unnormalized global solution $\widetilde u$ equals
$W-S_{\rm layer}\tau$ outside the tubes and $w_j+H_{a,j}h_j$ inside.
The layer-kernel bound gives, simultaneously on the exterior and its trace,
\[
 \|\widetilde u\|_{\infty,\mathcal E},\ \|h\|_\infty
 <45000000A_f=:M_{\mathrm{ext}}A_f.
\]
The harmonic maximum principle and the barrier give
$x|\widetilde u|\le(bM_{\mathrm{ext}}+65/32)A_f<M_{\mathrm{ext}}A_f$ inside the tubes.
Since $s_a\le1$ in the exterior, $\sup_Xs_a|\widetilde u|\le M_{\mathrm{ext}}A_f$.
{For the mean, use the unweighted bounds
$|H_{a,j}h_j|\le\|h\|_\infty\le M_{\mathrm{ext}}A_f$
and $|\widetilde u|\le M_{\mathrm{ext}}A_f$ on the exterior.
Only $w_j$ costs $(65/32)A_f/x$, giving
$|\int_X\widetilde u\,\dd V_a|
\le M_{\mathrm{ext}}A_fV_a+
(65/32)A_f\sum_j\int_{M_j}x^{-1}\,\dd V_a$.
}

The exact background volume is $V_a=(1-16\pi^2a^4)/4\ge1/5$.
Using \eqref{c0:weighted-local-volume},
\[
 \frac{\left|\int_X\widetilde u\,\dd V_a\right|}{V_a}
 \le M_{\mathrm{ext}}A_f+\frac{65}{32}\frac{13}{1000}(5)A_f
 <\left(M_{\mathrm{ext}}+\frac7{50}\right)A_f.
\]
Subtracting this mean from $\widetilde u$ gives $u=G_af$.
Finally, $s_a\le1$ yields
\[
 \sup_Xs_a|G_af|
 <\left(2M_{\mathrm{ext}}+\frac7{50}\right)A_f<10^8A_f,
\]
which proves \eqref{c0:main-bound}.
\end{proof}

All numerical comparisons above have displayed rational upper bounds.
The mass-frequency estimate, lattice counting, and source-energy argument
do not replace analytic estimates by numerical sampling. A full weighted
derivative estimate must additionally use the local Schauder constant
proved in Appendix~\ref{sch:sec:schauder}.


\section{The six-dimensional geometric Schauder certificate}
\label{sch:sec:schauder}

This section computes the local geometric constants for the EH$\times E$
background. All estimates are in real dimension six and are independent of
the resolution scale $a$. The weighted Schauder estimate proved here is local;
a global inverse estimate additionally requires a weighted zeroth-order bound.
A contraction in the boundary energy space does not supply that bound.

\subsection{Background, weights, and norms}
\label{sch:subsec:setup}

We use the quotient, radii, and background of Section~\ref{geom:section}.
The resolved normal factors are $\mathcal O_{\mathbb P^1}(-2)$, their
elliptic factors have periods $1,\ii$, and the pairing maps preserve the
radial coefficients in \eqref{geom:radial-eigenvalues}.
For the cutoff differentiation below, write $v(t)=\nu(t)$, where $\nu$
and $H$ are the step functions used in \eqref{geom:bump}.
Throughout this section, assume
\begin{equation}
 0<a\le2^{-10},\qquad \gamma=\frac12,\qquad \kappa=2^{-200}.
 \label{sch:eq:parameters}
\end{equation}
The background is exactly Eguchi--Hanson times the flat elliptic curve
for $r<r_1$ and flat for $r\ge c$. Its real metric is $g_a$.
The bounds \eqref{geom:bump-bounds}, positivity of $D'$, and
$a^4-4\beta_a J(r)\ge0$ will be used in the chart estimates.

Define a smooth full weight $s$ by setting $s=D(\rho)^{1/4}$ for $r\le b$,
and, on the flat annulus $b\le r\le2b$, setting
\[
 s=r+(1-r)H\!\left(\frac{r-b}{b}\right).
\]
Set $s=1$ elsewhere in the exterior. This full weight is the function
$s_a$ used in the main text. The open annuli are disjoint and the
cutoff is flat to all orders at its endpoints. Thus $s$ is globally smooth
and satisfies $a\le s\le1$.

For a positive weight $\varpi$, put
$\varpi_{pq}=\min(\varpi(p),\varpi(q))$ and
\[
 \mathcal P_\varpi=
 \{(p,q):0<d_{g_a}(p,q)\le\kappa\varpi_{pq}\}.
\]
{For tensor comparisons, restrict to positive weights $\varpi\le s$,
as holds for both weights used here.
Then $d(p,q)\le\kappa s(p)$ lies within the injectivity scale below.
}
Tensors are compared by parallel transport along minimizing geodesics.
Uniqueness at the distances used here is proved below. All tensor norms are
real Hilbert--Schmidt norms, and the same absolute-value estimates apply to
complex functions. Write $\operatorname{PT}_{q\to p}$ for parallel transport
and define
\begin{align}
 \|u\|_{\mathcal E_\varpi}
 &=
 \sum_{j=0}^2\sup_X\varpi^{j+1}|\nabla^j u|
 +\sup_{(p,q)\in\mathcal P_\varpi}
 \frac{\varpi_{pq}^{7/2}
 |\nabla^2u(p)-\operatorname{PT}_{q\to p}\nabla^2u(q)|}
 {d_{g_a}(p,q)^{1/2}},\\
 \|F\|_{\mathcal Y_\varpi}
 &=
 \sup_X\varpi^3|F|
 +\sup_{(p,q)\in\mathcal P_\varpi}
 \frac{\varpi_{pq}^{7/2}|F(p)-F(q)|}{d_{g_a}(p,q)^{1/2}},\\
 Z_s(u)&=\sup_Xs|u|,\qquad L_a=\frac12\Delta_{g_a}.
 \label{sch:eq:weighted-norms}
\end{align}
The subscript $s$ denotes the full weight. These exponents correspond to
the $\delta=-1$ weighted $C^{2,1/2}$ and $C^{0,1/2}$ norms.

\begin{proposition}[Explicit local weighted Schauder bound]
\label{sch:prop:schauder}
For every smooth real- or complex-valued function $u$, in the range
\eqref{sch:eq:parameters},
\begin{equation}
 \|u\|_{\mathcal E_s}
 \le2^{1060}\bigl(\|L_a u\|_{\mathcal Y_s}+Z_s(u)\bigr).
 \label{sch:eq:full-schauder}
\end{equation}
For $\varpi_{\rm cap}=\min(s,1/32)$, we also have the direct estimate
suited to the nonlinear analysis:
\begin{equation}
 \|u\|_{\mathcal E_s}
 \le2^{1060}\bigl(\|L_a u\|_{\mathcal Y_{\varpi_{\rm cap}}}+Z_s(u)\bigr).
 \label{sch:eq:capped-source-schauder}
\end{equation}
Both estimates hold for complex functions.
\end{proposition}

\subsection{Smooth EH core charts and a finite derivative certificate}
\label{sch:subsec:core}

At a center with $r\le2a$, choose the larger normal homogeneous coordinate
and write
\[
 Q=u^2=a^2\eta,\qquad \zeta=v/u,\qquad
 h=1+|\zeta|^2,\qquad x=|\eta|^2h^2,\qquad t=\sqrt{1+x}.
\]
Subtracting $a^2\log|Q|$ from the EH potential gives a smooth local
potential. The real line element is
\begin{equation}
 a^{-2}g_{\rm EH}
 =\frac{t}{h^2}|\dd\zeta|^2
 +\frac{h^2}{4t}
 \left|\dd\eta+\frac{2\eta\bar\zeta}{h}\dd\zeta\right|^2.
 \label{sch:eq:eh-metric}
\end{equation}
Its Hermitian matrix $K$ satisfies
\begin{align}
 K_{\eta\bar\eta}&=\frac{h^2}{4t},&
 K_{\eta\bar\zeta}&=\frac{h\bar\eta\zeta}{2t},&
 K_{\zeta\bar\zeta}&=\frac{t}{h}-\frac{|\zeta|^2}{h^2t},&
 \det_{\mathbb C}K&=\frac14.
 \label{sch:eq:eh-entries}
\end{align}
On the rectangle $|\eta|\le5,\ |\zeta|\le2$,
\[
 h\le5,\qquad \frac{t}{h}=\sqrt{h^{-2}+|\eta|^2}\le\sqrt{26},
 \qquad \operatorname{Tr}K\le\frac{25}{4}+\sqrt{26}<12.
\]
The determinant and trace bounds place both complex eigenvalues in
$[1/48,12]$. After adding the flat tangential eigenvalue, use the wider
common bounds
\begin{equation}
 m=2^{-8},\qquad M=2^6,\qquad m\Id\le\widehat g\le M\Id.
 \label{sch:eq:ellipticity}
\end{equation}
At the center, $|\eta_0|\le4,\ |\zeta_0|\le1$, and
$1\le s_p/a\le17^{1/4}<3$. With $s_p=s(p)$, the correctly scaled chart is
\begin{equation}
 (\eta,\zeta)=(\eta_0,\zeta_0)+(s_p/a)(y_1,y_2),
 \qquad w=w_0+s_p y_3,\qquad |y|<1/100.
 \label{sch:eq:core-chart}
\end{equation}
Consequently the tangential block of $\widehat g=s_p^{-2}\Phi_p^*g_a$
is exactly $|\dd y_3|^2$. In this chart,
$r/a=\sqrt{|\eta|h}\le5$, and $5a\le5/1024<r_1$.
The entire chart therefore lies in the exact EH core.
{At the divisor, the nonzero homogeneous coordinate is on
$\mathbb P^1$; the bundle chart includes $Q=0$.
Its radius-$5a$ normal image stays in a single tube lift,
and $s_p/100<1/2$ keeps the elliptic chart below the half-period.
}

The following third-order derivative bounds can be checked term by term.
A tuple $A=[A_0,A_1,A_2,A_3]$ bounds every real mixed component derivative
of the indicated order. Products satisfy
\begin{equation}
 (AB)_j\le\sum_{k=0}^j\binom jk A_kB_{j-k}.
 \label{sch:eq:jet-product}
\end{equation}
We may take
\[
 \begin{array}{c|rrrr}
  &0&1&2&3\\ \hline
 h&5&4&2&0\\
 |\eta|^2&25&10&2&0\\
 h^2&25&40&52&48\\
 x=|\eta|^2h^2&625&1250&2150&3000
 \end{array}
\]
Writing $x_1=1250,\ x_2=2150,\ x_3=3000$, the chain rule gives
\begin{align*}
 t&:\ [26,x_1,x_1^2+x_2,x_1^3+3x_1x_2+x_3],\\
 t^{-1}&:\ [1,x_1,2x_1^2+x_2,2x_1^3+6x_1x_2+x_3],\\
 h^{-1}&:\ [1,4,34,432],\qquad
 h^{-2}:\ [1,8,100,1680].
\end{align*}
For example, the first three scalar derivatives of $(1+x)^{-1/2}$ have
absolute values at most $1/2,3/4,15/8$, respectively, so the tuple for
$t^{-1}$ is valid. Also take
$\eta:[5,1,0,0]$, $\zeta:[2,1,0,0]$, 
$|\zeta|^2:[4,4,2,0]$. Substitution into \eqref{sch:eq:eh-entries} and
application of \eqref{sch:eq:jet-product} yield
\[
 \begin{array}{c|rrrr}
 K_{11}&25/4&15645/2&39139401/2&24608725762\\
 K_{12}&25&62575/2&78272543&98411415658\\
 K_{22}&30&6390&14174600&18009104700
 \end{array}
\]
This table is used only for derivative bounds; zeroth-order ellipticity
uses the sharper trace estimate above. The $j$th derivatives in
\eqref{sch:eq:core-chart} acquire a factor at most $3^j$.
Even allowing a further factor $2$ when passing from the Hermitian matrix
to the real matrix, the largest bound is
\[
 2\cdot27\cdot98411415658=5314216445532<2^{43}.
\]
We use the common component derivative bound $2^{60}$ below.

\subsection{Cutoff bounds, end charts, and uniform third-order control}
\label{sch:subsec:cutoff}

With $z=1/t$,
\[
 v'=e^{-z}z^2,\quad
 v''=e^{-z}(z^4-2z^3),\quad
 v'''=e^{-z}(z^6-6z^5+6z^4).
\]
Using $\sup_{z\ge0}z^me^{-z}=(m/e)^m$ and $e>2.7$,
\[
 |v'|\le1,\qquad |v''|\le9,\qquad |v'''|\le300.
\]
For $0\le t\le1$, we have $d=v(t)+v(1-t)\ge e^{-2}>1/8$, hence
\[
 |d'|\le2,\quad |d''|\le18,\quad |d'''|\le600,\qquad
 d^{-1}:\ [8,128,5248,345600].
\]
The product and quotient rules give
\begin{equation}
 |H'|\le136,\quad |H''|\le5576,\quad |H'''|\le367200.
 \label{sch:eq:cutoff-exact-majorants}
\end{equation}
Below we use the simpler bounds $2^8,2^{13},2^{19}$, respectively.

Set $B=\beta_a\psi$ and $\mathcal D=r\partial_r$.
The two transition supports are disjoint; on either transition the other
factor is identically $1$. Their length is $1/128$, and $128r\le4$
on their supports. Therefore
\begin{align}
 |\mathcal DB|&\le\frac{4\cdot2^8}{65}<16,\\
 |\mathcal D^2B|&\le\frac{16\cdot2^{13}+4\cdot2^8}{65}<2^{11},\\
 |\mathcal D^3B|&\le
 \frac{64\cdot2^{19}+48\cdot2^{13}+4\cdot2^8}{65}<2^{20}.
 \label{sch:eq:euler-bump}
\end{align}
For $a\le r_0\le b$, first use lifted quotient coordinates
$z=z_0+r_0y,\ w=w_0+r_0y_3$ on $|y|<1/10$, so that $r\ge0.9r_0$.
The normal chart is disjoint from its $-1$ translate. Put $H_n=D/r^4$
and let $P$ be the real orthogonal projector onto the complex radial line.
Then
\begin{equation}
 g_{\rm normal}
 =H_n^{1/2}\Id+\bigl[(1-B)H_n^{-1/2}-H_n^{1/2}\bigr]P.
 \label{sch:eq:end-metric}
\end{equation}
Since $a^4-4\beta_a J\ge0$,
\[
 1\le H_n\le1+(10/9)^4<3,\qquad
 (1-B)H_n^{-1/2}\ge\frac{64}{65\sqrt3}>\frac12,\qquad
 H_n^{1/2}<2.
\]
Thus \eqref{sch:eq:ellipticity} remains valid.
Differentiating $J$ gives
\[
 \mathcal D H_n=-4(H_n-1)-4B,\qquad
 |\mathcal DH_n|\le9,\quad
 |\mathcal D^2H_n|\le100,\quad
 |\mathcal D^3H_n|\le2^{14}.
\]
For a radial scalar function $F$, the following bounds apply componentwise:
\begin{align*}
 |\partial F|&\le|\mathcal DF|/r,\\
 |\partial^2F|&\le(|\mathcal D^2F|+3|\mathcal DF|)/r^2,\\
 |\partial^3F|&\le
 (|\mathcal D^3F|+9|\mathcal D^2F|+14|\mathcal DF|)/r^3.
\end{align*}
The resulting tuples in $y$ coordinates may be taken as
\[
 H_n:[3,16,256,65536],\qquad
 B:[1,32,4096,2097152].
\]
For $P=(z\otimes z+Jz\otimes Jz)/r^2$, the normalized numerator and
reciprocal denominator have bounds $[3,5,4,0]$ and $[2,4,32,256]$.
Their product gives $[6,22,144,1296]$; using $|P_{ij}|\le1$ at order zero,
we may take $P:[1,32,256,2048]$. We may further use
\[
 \begin{array}{c|rrrr}
 p=H_n^{1/2}&2&8&256&65536\\
 H_n^{-1/2}&1&8&512&65536\\
 q=(1-B)H_n^{-1/2}&1&64&8192&4194304\\
 q-p&3&72&8448&4259840
 \end{array}
\]
From \eqref{sch:eq:end-metric}, a third-order Leibniz bound for a metric
entry is
\[
 65536+4259840+3\cdot8448\cdot32
 +3\cdot72\cdot256+3\cdot2048
 =5197824<2^{23}.
\]
At the center $s_p/r_0\le2^{1/4}<6/5$. Passing to
$z=z_0+s_py,\ w=w_0+s_py_3$, $|y|<1/100$
multiplies derivatives of orders up to three by $(6/5)^j<2$,
so $2^{30}$ is already sufficient.

In the flat region outside the tubes, use an affine chart at scale $s_p$.
Its physical radius is at most $1/100$; its center is at least
$b-c=1/32$ from the nonflat region, so the chart remains flat.
Orbit separation and elliptic periods exceed the chart diameter.
The core and end charts overlap on $a\le r\le2a$ and cover the interior.
{The radius-$r_0/10$ end-chart ball is disjoint from its negative
and stays within radius $1.1b<2b$ of its fixed curve.
Outside the radius-$b$ tubes, fixed-locus involutions displace
centers by at least $2b=1/8$; their product has a third-coordinate
half-period shift. These and the unit lattice periods exceed
the chart diameter $2/100$, ensuring injectivity.
}
Thus every point has a holomorphic chart on the real ball $B_{1/100}$
satisfying
\begin{equation}
 \widehat g_p=s_p^{-2}\Phi_p^*g_a,\qquad
 m\Id\le\widehat g_p\le M\Id,\qquad
 \max_{\substack{1\le|\alpha|\le3\\i,j}}
 |\partial^\alpha\widehat g_{p,ij}|\le2^{60}.
 \label{sch:eq:chart-certificate}
\end{equation}
\subsection{Connection, curvature, convexity, and weight comparisons}
\label{sch:subsec:convexity}

In real dimension $6$, the Christoffel formula gives the component bound
\[
 |\Gamma^k_{ij}|\le9m^{-1}2^{60}<2^{72}.
\]
In Euclidean operator norm, use $|\Gamma|\le2^{80}$.
Differentiating and using
$\partial(g^{-1})=-g^{-1}(\partial g)g^{-1}$ gives
\[
 |\partial\Gamma\text{ component}|
 \le54m^{-2}(2^{60})^2+9m^{-1}2^{60}<2^{143}.
\]
Use the common operator bound $|D\Gamma|\le2^{160}$.
{Passing from components to Hilbert--Schmidt norms costs
$\sqrt{6^3}$ for $\Gamma$ and $\sqrt{6^4}=36$ for $D\Gamma$.
These factors fit within $2^{80}$ and $2^{160}$, respectively.
}
The curvature formula gives
\begin{equation}
 |\sec|\le2Mm^{-2}(|D\Gamma|+|\Gamma|^2)<2^{185}.
 \label{sch:eq:curvature}
\end{equation}
At a center in $B_{1/200}$, take a Euclidean initial velocity
$|v|\le\varepsilon_g=2^{-100}$. A bootstrap in the geodesic equation gives
$|z'|\le2\varepsilon_g$, with the path remaining in the chart.
Differentiating with respect to the initial velocity,
$Z(0)=0,\ Z'(0)=\Id$, and
\[
 |Z''|\le4\cdot2^{160}\varepsilon_g^2|Z|
          +4\cdot2^{80}\varepsilon_g|Z'|.
\]
{The velocity bootstrap improves because
$|z'-v|\le4\cdot2^{80}\varepsilon_g^2<\varepsilon_g$
and $|z-z(0)|\le2\varepsilon_g<1/200$.
Integrating the displayed $Z''$ bound likewise improves
$|Z|,|Z'|\le2$.
}
The bootstrap $|Z'|\le2,\ |Z|\le2$ yields
\[
 \|D_v\exp-\Id\|
 \le8(2^{80}\varepsilon_g+2^{160}\varepsilon_g^2)
 =8(2^{-20}+2^{-40})<\frac14.
\]
Integrating along a line segment in initial-velocity space proves
injectivity of the exponential map on this ball.
Any path leaving the chart has length at least $\sqrt m/200$.
A minimizing geodesic of length at most $2^{-106}$ has Euclidean initial
velocity at most
$m^{-1/2}2^{-106}=2^{-102}<\varepsilon_g$.
The scaled injectivity radius is therefore at least $2^{-106}$.
{ The derivative bound gives
$|\exp(v)-\exp(w)-(v-w)|\le|v-w|/4$.
Completeness supplies minimizing geodesics; at length at most
$2^{-106}$ they cannot exit the chart, and their initial velocities
lie in the injective ball. Thus short radial geodesics remain
unique minimizers.
}

To check strong convexity, consider a radial Jacobi field satisfying
$J(0)=0,J(r)=v\perp\partial_r$. The index form and Cauchy--Schwarz give
\begin{align*}
 \operatorname{Hess}r(v,v)&=
 \int_0^r\bigl(|D_tJ|^2-\langle R(J,\gamma')\gamma',J\rangle\bigr)\,\dd t,
 \\
 \int_0^r|J|^2\,\dd t&\le\frac{r^2}{2}\int_0^r|D_tJ|^2\,\dd t,
 \\
 |v|^2&\le r\int_0^r|D_tJ|^2\,\dd t.
\end{align*}
Combining this with \eqref{sch:eq:curvature}, and including the radial
direction, which contributes $1$, gives on scaled radius $3\kappa$,
\[
 \operatorname{Hess}(d_p^2/2)
 \ge[1-2^{185}(3\kappa)^2/2]g\ge g/2.
\]
The minimizing geodesic between two points of the $\kappa$-ball has length
at most $2\kappa$. The triangle inequality first keeps it in the
$3\kappa$-ball, and convexity of squared distance then keeps it in the
original ball. Hence the physical balls of radius $\kappa s_p$ are strongly
convex, and the parallel transport used in the norms is well defined.
{ Integration of the Jacobi equation gives the boundary term
$\langle D_tJ(r),J(r)\rangle$.
The injectivity estimate applies at both endpoints in the parent
chart, so the connector is unique before convexity is used.
The triangle inequality places it within the smooth $3\kappa$-ball;
convexity then confines it to the original ball.
}

Inside a tube, $s$ is the radial coordinate $x$, and
\[
 |\nabla s|=\sqrt{F_{\rm rad}}\le1,\qquad
 F_{\rm rad}=\frac{r^4(1-B)}{D}\le1.
\]
On the flat weight-transition annulus,
$|s'|\le1+b^{-1}\|H'\|_\infty\le1+16\cdot256<2^{13}$.
Consequently
\begin{equation}
 |\nabla s|\le2^{13},
 \qquad s(q)\in[s_p/2,2s_p]
 \quad\text{on the relevant local balls and their geodesics}.
 \label{sch:eq:weight-comparison}
\end{equation}
{The relevant neighborhoods satisfy $d(p,q)\le3\kappa s_p$, hence
$|s(q)-s_p|\le3\cdot2^{13}\kappa s_p<s_p/2$.
This includes the short geodesics and source charts, whose
physical path radius is at most $8\rho_*s_p<3\kappa s_p$.
}
\subsection{The Newton-kernel estimate in real dimension six}
\label{sch:subsec:euclidean}

The following Euclidean lemma is the analytic localization input.
Its explicit integral kernel and quantitative proof are included to make
the dependence transparent. The lemma depends only on real dimension $6$,
and is independent of the EH metric and the quotient group.

\begin{lemma}[Euclidean Schauder estimate with an explicit constant]
\label{sch:lem:euclidean}
Let $0<\gamma<1$ and $\Delta u=f$ distributionally on
$B_1\subset\mathbb R^6$, with
$u$ bounded and $f\in C^{0,\gamma}(B_1)$.
{ 
Here $u$ denotes the regular representative of its distributional
class, obtained below; the estimate concerns that representative.
}
Define the $C^{2,\gamma}$ norm as
$\sup|u|+\sup|Du|+\sup|D^2u|+[D^2u]_\gamma$,
and use Euclidean Hilbert--Schmidt norms for the gradient and Hessian. Then
\begin{equation}
 \|u\|_{C^{2,\gamma}(B_{1/2})}
 \le C_E(\gamma)
 \bigl(\|u\|_{\infty,B_1}+\|f\|_{C^{0,\gamma}(B_1)}\bigr),
 \qquad C_E(\gamma)=\frac{2^{24}}{\gamma(1-\gamma)}.
 \label{sch:eq:euclidean}
\end{equation}
The estimate also holds for complex functions.
\end{lemma}

\begin{proof}
Set $\sigma_5=\pi^3$ and use the Euclidean fundamental solution
\[
 \Gamma(z)=-\frac{|z|^{-4}}{4\sigma_5},\qquad
 K_{ij}(z)=D_{ij}\Gamma(z)
 =\frac{\delta_{ij}-6\nu_i\nu_j}{\sigma_5|z|^6},
 \quad \nu=z/|z|.
\]
Direct differentiation and spherical moments give
\[
 |K_{ij}|\le\frac6{\sigma_5|z|^6},\qquad
 |D_eK_{ij}|\le\frac{66}{\sigma_5|z|^7},\qquad
 \int_{S^5}K_{ij}(r\nu)\,\dd\nu=0.
\]
Let $\Pi$ be nearest-point projection onto $\overline B_1$,
let $\eta(r)=1$ for $r\le1$, $\eta(r)=2-r$ for $1<r<2$,
and $\eta(r)=0$ otherwise, and set $F(z)=\eta(|z|)f(\Pi z)$.
Then $\|F\|_\infty\le\|f\|_\infty$ and
$[F]_\gamma\le[f]_\gamma+\|f\|_\infty$.
For $V=\Gamma*F$, radial integration yields
\[
 \|V\|_{\infty,B_1}\le\frac98\|f\|_\infty,\qquad
 \max_i\|V_i\|_{\infty,B_1}\le3\|f\|_\infty,\qquad
 \max_{ij}\|V_{ij}\|_{\infty,B_1}
 \le\frac16\|f\|_\infty+\frac{24}{\gamma}[F]_\gamma.
\]
The Hessian representation contains the distributional term
$\delta_{ij}F(x)/6$ and
$\operatorname{p.v.}\int K_{ij}(x-z)F(z)\,\dd z$.
{Integration by parts outside a small ball gives
$\sigma_5^{-1}\int_{S^5}\nu_i\nu_jF(x)\,\dd\nu=\delta_{ij}F(x)/6$.
Cancellation after subtracting $F(x)$ leaves the integrable
majorant $r^{\gamma-1}\,\dd r$. Mollification and these bounds
justify passage to a continuous H\"older Hessian.
}
For two points at distance $d$, split the integrals inside and outside
$B_{2d}(x)$. The near contribution is at most
$(30/\gamma)d^\gamma[F]_\gamma$; the mean-value derivative bound for the
kernel controls the far contribution by
$8448d^\gamma[F]_\gamma/(1-\gamma)$.
{ The near balls lie in $B_{2d}(x)$ and $B_{3d}(y)$, costing
$6(2^\gamma+3^\gamma)d^\gamma/\gamma\le30d^\gamma/\gamma$.
On the far region the comparison segment stays at distance
at least $r/2$; its radial integral is
$66\cdot2^7d\int_{2d}^\infty r^{\gamma-2}\,\dd r
\le8448d^\gamma/(1-\gamma)$.
Both bounds multiply $[F]_\gamma$.
}
Cancellation for the translated ball costs at most
$6\log3\,d^\gamma[F]_\gamma$. Thus
\[
 \max_{ij}[V_{ij}]_\gamma
 \le\left(\frac{30}{\gamma}+\frac{8448}{1-\gamma}+13\right)[F]_\gamma.
\]
{Since $B_d(y)\subset B_{2d}(x)\subset B_{3d}(y)$, cancellation
leaves a shell costing $6\int_d^{3d}\dd r/r=6\log3$.
Outer translation errors vanish under expanding centered truncations.
The distributional term adds $1/6$, and $6\log3+1/6<13$.
}
{
More explicitly, for nonnegative standard mollifiers set
$F_\varepsilon=F*\varphi_\varepsilon$ and
$V_\varepsilon=\Gamma*F_\varepsilon$. Then
$F_\varepsilon\to F$ uniformly and
$[F_\varepsilon]_\gamma\le[F]_\gamma$. The preceding kernel
estimates give locally equibounded and equicontinuous Hessians
$D^2V_\varepsilon$. Their subsequential uniform limits equal
$D^2V$ distributionally, hence define its continuous representative;
the same H\"older bound passes to the limit. Convergence in the
full $C^{0,\gamma}$ norm is not required.
}
By Weyl's lemma, $W=u-V$ is smooth and harmonic, with
$\|W\|_{\infty,B_1}\le
W_0:=\|u\|_{\infty,B_1}+(9/8)\|f\|_\infty$.
The first three derivatives of the normalized Poisson kernel at the
center of a radius-$\rho$ ball are
\[
 \frac{6\nu_i}{\rho},\qquad
 \frac{48\nu_i\nu_j-8\delta_{ij}}{\rho^2},\qquad
 \frac{480\nu_i\nu_j\nu_k
 -60(\delta_{ij}\nu_k+\delta_{ik}\nu_j+\delta_{jk}\nu_i)}{\rho^3}.
\]
At points of $B_{1/2}$, let $\rho\uparrow1/2$ to obtain component bounds
$12W_0,224W_0,5280W_0$.
Convexity and the unit diameter of the ball give
$\max_{ij}[W_{ij}]_\gamma\le15840W_0$.
The Hilbert--Schmidt norms of the gradient and Hessian are at most six
times their respective largest component bounds. Summing the displayed
estimates already gives $2^{19}/[\gamma(1-\gamma)]$;
\eqref{sch:eq:euclidean} leaves a wider numerical margin.
All integral estimates use absolute values, so the same constant covers
complex functions.
\end{proof}

Rescaling gives the estimate on $B_R$, with derivatives weighted by
$R,R^2,R^{2+\gamma}$ and source terms by $R^2,R^{2+\gamma}$.
We use $\gamma=1/2$, hence $C_E=2^{26}$.

\subsection{Quantitative coefficient freezing and absorption}
\label{sch:subsec:freezing}

Real coordinate functions of a holomorphic chart are harmonic for a
K\"ahler metric, so
\[
 \widehat L=A:D^2,\qquad A=\tfrac12\widehat g^{-1},
\]
with no first-order drift. From \eqref{sch:eq:chart-certificate},
\[
 \lambda\Id\le A\le\Lambda\Id,\qquad
 \lambda=2^{-7},\quad\Lambda=2^7,\quad[A]_{1/2}\le2^{80}.
\]
The last bound follows from
$\|DA\|_{\rm HS}\le\frac12m^{-2}6\sqrt6\,2^{60}<2^{80}$,
and the chart has diameter less than $1$, so the Lipschitz bound also
controls the $1/2$-H\"older seminorm.

On a Euclidean domain $\mathcal O$, write
$d(x)=\operatorname{dist}(x,\partial\mathcal O)$,
$d_{xy}=\min(d(x),d(y))$ and define
\begin{align*}
 N(v)&=\|v\|_\infty+\sup d|Dv|+\sup d^2|D^2v|
       +\sup_{x\ne y}
       \frac{d_{xy}^{5/2}|D^2v(x)-D^2v(y)|}{|x-y|^{1/2}},\\
 F^*(F)&=\sup d^2|F|+
       \sup_{x\ne y}\frac{d_{xy}^{5/2}|F(x)-F(y)|}{|x-y|^{1/2}}.
\end{align*}
Apply Lemma \ref{sch:lem:euclidean} on $B_{d(x)/2}(x)$.
The localization factors for the gradient, Hessian, short-pair H\"older
term, and long-pair H\"older term are $2,4,8,64$, respectively.
The two pair regimes contribute the maximum $64$ to the same seminorm;
including the zeroth-order term gives $1+2+4+64=71$ and hence
$N(v)\le71C_E(\|v\|_\infty+F^*(\Delta v))$.
{ On $B_{d(x)/2}(x)$ the boundary distance is at least $d(x)/2$,
so its scaled source norm is bounded by $F^*$.
Center estimates cost $2$ and $4$; pairs separated by at most
$d_{xy}/4$ cost $2^{5/2}<8$.
Other pairs cost at most $16<64$ by endpoint Hessian bounds.
The pair regimes share one supremum, so their costs are maximized.
Smoothness on larger charts ensures finiteness before absorption.
}

For a constant positive-definite matrix $A_0$, set $y=A_0^{-1/2}x$.
Boundary distances lie between $d_x/\sqrt\Lambda$ and $d_x/\sqrt\lambda$.
With $\mathfrak r=\Lambda/\lambda$, tensor transformation and distance
comparison give
$N_x\le\mathfrak r^{5/4}N_y$,
$F_y^*\le\max(1,\lambda^{-1})\mathfrak r^{1/4}F_x^*$.
Consequently
\[
 N(v)\le C_A\bigl(\|v\|_\infty+F^*(A_0:D^2v)\bigr),\qquad
 C_A=71\cdot2^{26}\max(1,\lambda^{-1})
 (\Lambda/\lambda)^{3/2}=71\cdot2^{54}<2^{61}.
\]
{ For $N_x$, the Hessian, pair denominator and boundary weight cost
$\lambda^{-1}$, $\lambda^{-1/4}$ and $\Lambda^{5/4}$,
giving $\mathfrak r^{5/4}$; lower derivatives cost less.
For $F_y^*$, the supremum and seminorm cost
$\lambda^{-1}$ and
$\lambda^{-5/4}\Lambda^{1/4}
=\lambda^{-1}\mathfrak r^{1/4}$.
The distance comparisons hold on the transformed ellipsoid.
}
On $B_{\rho_*}$, freeze $A_0=A(0)$, where
\begin{equation}
 \rho_*=2^{-288}.
 \label{sch:eq:freezing-radius}
\end{equation}
The product rule for the boundary-weighted norms gives
\[
 F^*((A_0-A):D^2v)
 \le2[A]_{1/2}\sqrt{\rho_*}\,N(v).
\]
The two terms come from the coefficient supremum times the Hessian
H\"older seminorm, and the coefficient H\"older seminorm times the Hessian
supremum; $d_{xy}\le\rho_*$ supplies the stated factor. We have
$2C_A[A]_{1/2}\sqrt{\rho_*}\le
2\cdot2^{61}\cdot2^{80}\cdot2^{-144}=1/4$.
{Writing $H_{\rm sup}=\sup d^2|D^2v|$ and $H_{\rm pair}$ for
the Hessian seminorm in $N$, the source supremum costs
$[A]_{1/2}\sqrt{\rho_*}H_{\rm sup}$.
The product differences cost
$[A]_{1/2}\sqrt{\rho_*}(H_{\rm pair}+H_{\rm sup})$:
choose the unchanged Hessian at the larger-distance endpoint,
using $d_{xy}^{5/2}/d_{\max}^2\le\sqrt{\rho_*}$.
Thus the factor two includes the source supremum.
}
Absorption yields, on $B_{\rho_*/2}$,
\begin{equation}
 \|v\|_{C^{2,1/2}}
 \le2C_A(2/\rho_*)^{5/2}
 \left(\|v\|_{\infty,B_{\rho_*}}+
 \rho_*^2\|\widehat Lv\|_{C^{0,1/2}(B_{\rho_*})}\right),
 \qquad2C_A(2/\rho_*)^{5/2}<2^{785}.
 \label{sch:eq:local-estimate}
\end{equation}
\subsection{Covariant derivatives and parallel-transport H\"older norms}
\label{sch:subsec:transfer}

Put $F=\|L_au\|_{\mathcal Y_s}$, $Z=Z_s(u)$, and
$v_p=s_pu\circ\Phi_p$. The exact scaled equation is
\[
 \widehat L v_p=s_p^3(L_au)\circ\Phi_p.
\]
The physical radius of $B_{\rho_*}$ is at most
$8\rho_*s_p<\kappa s_p$.
Every pair in this chart belongs to the pair set defining the norm.
{The chart diameter is $16\rho_*s_p$ and its minimum weight
is at least $s_p/2$; the required ratio is therefore
$32\rho_*=2^{-283}<\kappa$.
}
Using \eqref{sch:eq:weight-comparison} and the scalar H\"older comparison
factor $M^{1/4}<4$,
{ 
\[
 \|s_p^3(L_au)\circ\Phi_p\|_{C^{0,1/2}}
 \le8F+4\cdot2^{7/2}F\le64F,\qquad
 \|v_p\|_\infty\le2Z.
\]
}
All four coordinate norm terms in \eqref{sch:eq:local-estimate} are
therefore bounded by
\begin{equation}
 H_0=2^{792}(F+Z).
 \label{sch:eq:coordinate-bound}
\end{equation}
We retain $64F$ because the seminorm uses the minimum endpoint weight;
the smaller source comparison for a center-based seminorm is not used.

{Set $h_p=D^2v_p-\Gamma Dv_p$. Then
$|h_p|_{\widehat g}\le m^{-1}(1+2^{80})H_0\le2^{89}H_0$.
}
The physical Hessian acquires a further factor $s_p^{-3}$, canceled by
its weight. We can also prove the stronger center-based seminorm estimate:
take both endpoints in the $\kappa s_x$-ball centered at $x$ and split at
\[
 d_{\rm phys}=\delta_*s_x,\qquad \delta_*=2^{-300}.
\]
For a short pair, use the chart centered at the first endpoint $p$.
Its scaled geodesic length is at most $2\delta_*$, and its Euclidean
displacement is at most
$m^{-1/2}2\delta_*=2^{-295}<\rho_*/2$.
The entire geodesic therefore stays in $B_{\rho_*/2}$.
{Up to a first exit, the same metric bound controls total Euclidean
variation by $2^{-295}<\rho_*/2$.
Exiting $B_{\rho_*/2}$ would require variation at least $\rho_*/2$,
so the comparison holds along the entire connector.
}
{The coordinate H\"older bound is
$[h_p]_{1/2}\le(1+2^{80}+2^{160})H_0\le2^{162}H_0$.
}
Distance conversion costs $m^{-1/4}=4$ and tensor-norm conversion costs
$m^{-1}=256$, so the coordinate-difference contribution to the metric
tensor H\"older bound is at most
$2^{172}H_0$.

Let $B=h_p(q)$ be a coordinate-constant tensor, so
$|B|_{\rm Eucl}\le2^{81}H_0$.
For a unit vector $w$ in the scaled metric,
\[
 |\widehat\nabla_wB|_{\widehat g}
 \le2\cdot2^{80}m^{-3/2}|B|_{\rm Eucl}\le2^{174}H_0.
\]
Integrate along the minimizing geodesic using isometric parallel
transport, then divide by the square root of its length.
Since the scaled length is less than $1$, the cost is still at most
$2^{174}H_0$. Adding both terms and multiplying by
$(s_x/s_p)^{7/2}<16$ gives a short-pair cost at most $2^{179}H_0$.

For a long pair, the endpoint physical Hessian norms are at most
$2^{92}H_0s_x^{-3}$, parallel transport is isometric, and
$d_{\rm phys}\ge\delta_*s_x$. The weighted H\"older cost is at most
$2^{93}H_0/\sqrt{\delta_*}=2^{243}H_0$.
Adding the zeroth-, first-, and second-order terms and the H\"older term
costs at most $2^{245}H_0$. By \eqref{sch:eq:coordinate-bound}, the result is
$2^{1037}(F+Z)$. The common constant $2^{1060}$ proves
\eqref{sch:eq:full-schauder}.

\subsection{Capped-weight source comparison and the conditional inverse estimate}
\label{sch:subsec:cap}

Let $\varpi_{\rm cap}=\min(s,1/32)$. Then
$a\le\varpi_{\rm cap}\le1/32$ and $s/\varpi_{\rm cap}\le32$.
The target pair set and the target weights are smaller, hence
$\|u\|_{\mathcal E_{\varpi_{\rm cap}}}\le\|u\|_{\mathcal E_s}$.
The full and capped source norms also have different pair sets.
Their H\"older seminorms therefore cannot be compared globally merely by
comparing the weights.

The proof only uses the tiny source chart $B_{\rho_*}$.
On that chart $\min\varpi_{\rm cap}\ge s_p/64$, and every pair satisfies
\[
 d(p,q)\le16\rho_*s_p
 \le1024\rho_*\min(\varpi_{\rm cap}(p),\varpi_{\rm cap}(q))
 <\kappa\min(\varpi_{\rm cap}(p),\varpi_{\rm cap}(q)),
\]
because $1024\rho_*=2^{-278}<2^{-200}=\kappa$.
Thus all source pairs in the chart belong to the capped pair set.
For the same pair, $32^{7/2}<2^{18}$ is only the weight-comparison cost,
not the complete coordinate source constant. Including scalar H\"older
comparison gives
{\begin{align}
 \|s_p^3 f\circ\Phi_p\|_{C^{0,1/2}(B_{\rho_*})}
 &\le\left(64^3+4\cdot64^{7/2}\right)
       \|f\|_{\mathcal Y_{\varpi_{\rm cap}}}\\
 &\le2^{24}\|f\|_{\mathcal Y_{\varpi_{\rm cap}}}.
 \label{sch:eq:cap-source-local}
\end{align}
}
By \eqref{sch:eq:local-estimate}, we may take the coordinate bound
$H_{0,\rm cap}=2^{810}
(\|L_au\|_{\mathcal Y_{\varpi_{\rm cap}}}+Z_s(u))$.
Multiplying by the preceding transfer factor $2^{245}$ gives $2^{1055}$.
The common constant $2^{1060}$ therefore proves
\eqref{sch:eq:capped-source-schauder}, completing the proof of
Proposition \ref{sch:prop:schauder}.

\begin{corollary}[Global inverse constant from an independent zeroth-order certificate]
\label{sch:cor:inverse}
Suppose the global zero-mean inverse $G_a$ independently satisfies
\[
 Z_s(G_af)\le M_0\sup_Xs^3|f|,\qquad M_0\le10^8.
\]
Then
\begin{equation}
 \|G_af\|_{\mathcal E_{\varpi_{\rm cap}}}
 \le2^{1120}\|f\|_{\mathcal Y_{\varpi_{\rm cap}}}.
 \label{sch:eq:conditional-inverse}
\end{equation}
\end{corollary}
\begin{proof}
$\sup s^3|f|\le32^3\sup\varpi_{\rm cap}^3|f|
\le2^{15}\|f\|_{\mathcal Y_{\varpi_{\rm cap}}}$.
Apply \eqref{sch:eq:capped-source-schauder} and
$\mathcal E_{\varpi_{\rm cap}}\le\mathcal E_s$ to obtain
\[
 \|G_af\|_{\mathcal E_{\varpi_{\rm cap}}}
 \le2^{1060}(1+2^{15}M_0)
       \|f\|_{\mathcal Y_{\varpi_{\rm cap}}}
 \le2^{1120}\|f\|_{\mathcal Y_{\varpi_{\rm cap}}}.
\]
\end{proof}

This appendix proves the local geometric bounds, Euclidean freezing and
absorption, and covariant H\"older conversion. The zeroth-order hypothesis
of Corollary \ref{sch:cor:inverse} requires an independent global kernel and
transmission argument. No unknown Green kernel or unknown global spectrum
is used as an explicit input here, and the local Schauder estimate alone
does not prove convergence of the nonlinear Calabi--Yau series.

\end{document}